\documentclass[11pt,a4paper]{amsart}
\usepackage{amsmath,amsthm,amsfonts,amscd,amssymb,enumitem, graphicx,latexsym,mathrsfs,stmaryrd,tikz-cd,
bm,mathtools}

\usepackage{hyperref}

\usepackage{cleveref}

\newtheorem{theorem}{Theorem}[section]
\newtheorem{corollary}[theorem]{Corollary}
\newtheorem{lemma}[theorem]{Lemma}
\newtheorem{proposition}[theorem]{Proposition}

\theoremstyle{definition}
\newtheorem{definition}[theorem]{Definition}
\newtheorem{remark}[theorem]{Remark}

\newtheorem{example}[theorem]{Example}

\newtheorem*{definition*}{Definition}

\newcounter{theoremintro}

\numberwithin{equation}{section}

\allowdisplaybreaks

\newcommand{\mc}[1]{\mathcal{#1}}
\newcommand{\bb}[1]{\mathbb{#1}}
\def\bh{\bb{B}(\mc{H})}

\def\acts{\curvearrowright}

\def\her{\mathrm{her}}
\def\supp{\mathrm{supp}}

\def\sub{\subset}

\def\precsimd{\precsim_\mathrm{diag}}
\allowdisplaybreaks
\newcommand\blfootnote[1]{%
  \begingroup
  \renewcommand\thefootnote{}\footnote{#1}%
  \addtocounter{footnote}{-1}%
  \endgroup
}
\title{Diagonal comparison of Ample \texorpdfstring{$\mathrm{C}^\ast$}{C*}-diagonals}
\author{Grigoris Kopsacheilis}
\address{Grigoris Kopsacheilis, Mathematisches Institut, Fachbereich Mathematik und Informatik der
Universit\"at M\"unster, Einsteinstrasse 62, 48149 M\"unster, Germany.}
\email{gkopsach@uni-muenster.de}
\urladdr{https://sites.google.com/view/gkopsach}

\author{Wilhelm Winter}
\address{Wilhelm Winter, Mathematisches Institut, Fachbereich Mathematik und Informatik der
Universit\"at M\"unster, Einsteinstrasse 62, 48149 M\"unster, Germany.}
\email{wwinter@uni-muenster.de}
\begin{document}
\begin{abstract}
We introduce \emph{diagonal comparison}, a regularity property of diagonal pairs where the sub-$\mathrm{C}^\ast$-algebra has totally disconnected spectrum, and establish its equivalence with the concurrence of strict comparison of the ambient $\mathrm{C}^\ast$-algebra and dynamical comparison of the underlying dynamics induced by the partial action of the normalisers. As an application, we show that for diagonal pairs arising from principal minimal transformation groupoids with totally disconnected unit space, diagonal comparison is equivalent to tracial $\mc{Z}$-stability of the pair and that it is implied by finite diagonal dimension.

In-between, we show that any projection of the diagonal sub-$\mathrm{C}^\ast$-algebra can be uniformly tracially divided, and explore a property of conditional expectations onto abelian sub-$\mathrm{C}^\ast$-algebras, namely containment of every positive element in the hereditary subalgebra generated by its conditional expectation. We show that the expectation associated to a $\mathrm{C}^\ast$-pair with finite diagonal dimension is always hereditary in that sense, and we give an example where this property does not occur.
\end{abstract}
\maketitle
\blfootnote{\textit{Date}: \today .}

\section*{Introduction}

\renewcommand{\thetheorem}{\Alph{theorem}}
\setcounter{theorem}{0}

\noindent Topological dynamics and operator algebras are intricately related, and developments in the one field frequently inspire advances in the other. In the seminal work of Kerr \cite{Ker20}, structural properties of topological dynamics were identified as conceptual analogues of the regularity properties of $\mathrm{C}^\ast$-algebras which appeared in the Toms--Winter conjecture. Even though similarities between the two pictures have become apparent due to the results in \cite{Ker20} and \cite{KerSza20}, parts of this analogy remain somewhat obscure, thus leading to the definition and study of regularity properties of \emph{pairs} of $\mathrm{C}^\ast$-algebras \cite{LiLiaWin23, LiaTik22, KopLiaTikVac24}. In this paper we aim to build further on the simultaneous study of structure and regularity through the lens of sub-$\mathrm{C}^\ast$-algebras by introducing a comparison property for \emph{Cartan pairs} with the unique extension property (i.e.\ $\mathrm{C}^\ast$-\emph{diagonal} pairs).

The Toms--Winter conjecture  -- now almost a theorem -- is the statement that, for a unital, simple, separable, nuclear, infinite-dimensional $\mathrm{C}^\ast$-algebra, the properties of finite \emph{nuclear dimension}, $\mathcal{Z}$-\emph{stability}, and \emph{strict comparison} are in fact equivalent. It originated in \cite{TomWin09} in the context of Elliott's programme to classify nuclear $\mathrm{C}^\ast$-algebras and we refer the reader to \cite{Win18, Whi23} for further details and a historical account of the subject, and the current status of the conjecture.

On the side of dynamics, for a free minimal action $G \curvearrowright X$ of a countable, discrete, amenable group on a compact, metrisable space by homeomorphisms, Kerr introduced in \cite{Ker20} as respective counterparts to the properties appearing in the Toms--Winter conjecture the structural properties of \emph{finite tower dimension}, \emph{almost finiteness}, and \emph{dynamical comparison}. The main results in that work established a systematic approach for constructing crossed products that fall in the scope of classification theory, as finite tower dimension of $G \curvearrowright X$ implies that the reduced crossed product $\mathrm{C}^\ast$-algebra $C(X) \rtimes_\mathrm{r} G$ has finite nuclear dimension \cite[Theorem~6.2]{Ker20} and almost finiteness in turn implies that $C(X) \rtimes_\mathrm{r} G$ is $\mathcal{Z}$-stable \cite[Theorem~12.4]{Ker20}. 

The aforementioned implications were pushed forward in \cite{LiLiaWin23} and \cite{LiaTik22} respectively, using the framework of sub-$\mathrm{C}^\ast$-algebras. With the development of the theory of \emph{diagonal dimension}, H.\ Liao, K.\ Li and the second-named author obtained upper, and also \emph{lower} bounds of the diagonal dimension of $(C(X) \subset C(X)\rtimes_\mathrm{r} G)$ in terms of the tower dimension of the underlying action \cite[Theorem~5.4]{LiLiaWin23}, indicating that the structure of the action provided by finiteness of its tower dimension not only is strong enough to induce regularity for the crossed product, but it can also be recovered from this regularity, as long as this is witnessed by the sub-$\mathrm{C}^\ast$-algebra $(C(X) \subset C(X)\rtimes_\mathrm{r} G)$. Analogously, Liao and Tikuisis defined in \cite{LiaTik22} \emph{tracial} $\mathcal{Z}$-\emph{stability for a sub}-$\mathrm{C}^\ast$-\emph{algebra} $(D \subset A)$ (based on the Hirshberg--Orovitz tracial $\mathcal{Z}$-stability \cite{HirOro13}) and, in conjunction with \cite{KopLiaTikVac24}, showed that almost finiteness of $G \curvearrowright X$ is in fact equivalent to tracial $\mathcal{Z}$-stability of $(C(X) \subset C(X)\rtimes_\mathrm{r} G)$. Note that the sub-$\mathrm{C}^\ast$-algebras $(D \subset A)$ appearing here are in fact \emph{diagonal pairs}, i.e.\ Cartan pairs with the unique extension property, meaning that $D$ is a maximal abelian $^\ast$-subalgebra of $A$, the \emph{normalisers} $\mc{N}_A(D)=\{ v\in A: vDv^* \cup v^*Dv \subset D\}$ span a dense subspace of $A$ (\emph{regularity}), there is a faithful conditional expectation $A \to D$, and each pure state on $D$ extends uniquely to $A$ \cite{Ren08, Kum86}.

As for dynamical comparison, things are more perplexed. So far there is no example of an action where dynamical comparison fails, and in fact it is known to be automatic in an amplitude of cases, mostly ensured by geometric features of the acting group and the covering dimension of the space on which the action is implemented \cite{DowZha23, KerSza20, Nar24, Nar22, NarPet24}. It is clear that, in general, dynamical comparison alone does not suffice to conclude regularity on the $\mathrm{C}^\ast$-side for the associated crossed product, as was demonstrated in \cite{GioKer10, Nar22}. However, for actions on spaces of finite covering dimension (and in particular zero-dimensional spaces, which comprise the focal case in this context, see \cite[Theorem~7.6]{KerSza20}), dynamical comparison is in fact equivalent to almost finiteness \cite[Theorem~6.1]{KerSza20}, which in turn ensures $\mathrm{C}^\ast$-regularity of the crossed product.

Focusing on diagonal pairs where the sub-$\mathrm{C}^\ast$-algebra has zero-dimensional spectrum (which are usually called \emph{ample}), in order to understand better how dynamical comparison -- in its own right -- relates to regularity on the $\mathrm{C}^\ast$-side, we introduce the concept of \emph{diagonal comparison} of a pair $(D \subset A)$. 

\begin{definition}\label{intro def diag below and diag comp}
Let $(D\sub A)$ be a unital, ample $\mathrm{C}^*$-diagonal pair and let $\Phi\colon A\to D$ denote the associated conditional expectation. For positive elements $a,b\in A_+$, we say that $a$ is \textit{diagonally below} $b$ and write $a\precsimd b$, when for any $\varepsilon>0$ there exist elements $u,w,v\in A$ such that
\begin{enumerate}[label=(\roman*)]
\item\label{itm1intromaindef} $\|a-uwvbv^*w^*u^*\|<\varepsilon$,
\item\label{itm2intromaindef} $\supp(\Phi(u^*u))\sub\supp(\Phi(a))$,\footnote{For $d\in D$ we use $\supp(d)$ to denote the open support of $d$, i.e.\ the set of points in the spectrum of $D$ where $d$ does not vanish.}
\item\label{itm3intromaindef} $\supp(\Phi(vv^*))\sub\supp(\Phi(b))$,
\item\label{itm4intromaindef} $w\in\mc{N}_A(D)$.
\end{enumerate}
If the set of normalised quasitraces $\mathrm{QT}(A)$ is nonempty, we say that $(D\sub A)$ has \textit{diagonal comparison} when, for any $n\in\bb{N}$ and any $a,b\in (A\otimes M_n)_{+}$, if $d_\tau(a)<d_\tau(b)$ for all $\tau\in\mathrm{QT}(A)$, then $a\precsimd b$ in $(D\otimes D_n\sub A\otimes M_n)$.\footnote{Here $d_\tau(.)$ denotes the \emph{dimension function} associated to $\tau$, see  \cite[Definition~II.6.8.13]{Bla06}.}
\end{definition}

The diagonally below relation described in \Cref{intro def diag below and diag comp} is a special form of Cuntz subequivalence (condition \ref{itm1intromaindef} alone would imply that $a$ is Cuntz subequivalent to $b$). The novelty of this definition lies in the involved nature of the elements implementing the subequivalence; we require that these factorise as products of three elements, with the middle term, $w$, being a normaliser and hence imposing a dynamical relation (condition \ref{itm4intromaindef}). However, it is crucial to be able to compare elements from the positive cone of the \emph{ambient} $\mathrm{C}^\ast$-algebra and not just the Cartan sub-$\mathrm{C}^\ast$-algebra, which is why the auxiliary peripheral terms $u$ and $v$ appear in the factorisation. Conditions \ref{itm2intromaindef} and \ref{itm3intromaindef} then ensure that the normaliser $w$ is chosen in a `tight' manner.

The definition described above is well-motivated due to the fact that for an action $G \curvearrowright X$ on a zero-dimensional space, dynamical subequivalence of clopen subsets \cite[Definition~3.1, Proposition~3.5]{Ker20} is the same as Cuntz subequivalence of the respective indicator functions, implemented by normalisers of $(C(X) \subset C(X)\rtimes_\mathrm{r} G)$ (\Cref{rmk:dynamicalsubequivalencenormaliser}; cf.\ \cite[Proposition~2.9]{LiaTik22}). This observation also leads to a natural generalisation of the notion of dynamical comparison from topological dynamics to ample diagonal pairs $(D \subset A)$: we say that the pair has \emph{dynamical comparison} when, for any projections $p$ and $q$ in $D$, $\tau(p) <\tau(q)$ for all traces $\tau$ on $A$ implies that $p$ is Murray-von Neumann subequivalent to $q$ via a partial isometry that is also a normaliser of $(D \subset A)$ (\Cref{def:dynbelow}, \Cref{def:dyncomp}; cf.\ \cite[Definition~2.1]{AraBonBosLi23}).

For technical reasons we focus on $\mathrm{C}^\ast$-algebras $A$ with the property that $\mathrm{QT}(A)=\mathrm{T}(A)$, i.e.\ all quasitraces are traces. It is a long-standing open problem whether this is actually automatic; it is indeed true when $A$ is exact, due to the celebrated result of Haagerup \cite{Hag14}.

With \Cref{intro def diag below and diag comp} at hand, we are able to characterise the concurrence of dynamical comparison of an ample diagonal pair and strict comparison of the ambient $\mathrm{C}^\ast$-algebra (i.e.\ the property of strict inequality on dimension functions detecting Cuntz subequivalence of positive elements; cf.\ \Cref{def:strictcomp}), which is the content of the main theorem of this paper.

\begin{theorem}\label{intro main thm}
Let $(D \subset A)$ be a unital, ample $\mathrm{C}^\ast$-diagonal pair with hereditary associated conditional expectation,\footnote{A conditional expectation $A\to D$ is called \emph{hereditary} when each positive element of $A$ lies in the hereditary sub-$\mathrm{C}^\ast$-algebra generated by its image through the conditional expectation (cf.\ \Cref{def:hermap}).} and assume that $A$ is simple, separable and that $\mathrm{QT}(A)=\mathrm{T}(A)$ is nonempty. The following are equivalent:
\begin{enumerate}[label=\normalfont(\roman*)]
\item\label{itm1intromainthm} $(D \subset A)$ has dynamical comparison and $A$ has strict comparison.
\item\label{itm2intromainthm} $(D \subset A)$ has diagonal comparison.
\end{enumerate}
\end{theorem}

As a corollary, we obtain a connection of diagonal comparison with tracial $\mc{Z}$-stability and finiteness of diagonal dimension for $\mathrm{C}^\ast$-pairs arising from topological dynamics.

\begin{corollary}\label{intro cor diag comp and trac js stab}
Let $G \acts X$ be a free minimal action of a countable amenable group on a compact, zero-dimensional metrisable space. Consider the following statements:
\begin{enumerate}[label=\normalfont(\roman*)]
\item\label{itc1intro} $\dim_{\mathrm{diag}}(C(X)\subset C(X)\rtimes_\mathrm{r} G) <\infty$.
\item\label{itc2intro} $(C(X) \subset C(X)\rtimes_\mathrm{r} G)$ is tracially $\mathcal{Z}$-stable.
\item\label{itc3intro} $(C(X) \subset C(X)\rtimes_\mathrm{r} G)$ has diagonal comparison.
\end{enumerate}
Then \ref{itc1intro}$\implies$\ref{itc2intro}$\iff$\ref{itc3intro}.
\end{corollary}

Zero-dimensionality of the spectrum of $D$ is crucial for the proof of \Cref{intro main thm}, as it provides us with a versatile divisibility toolkit. In particular, using an elementary combinatorial argument similar to \cite[Section~1]{Zha91}, we see that for an ample diagonal pair $(D \subset A)$, if $A$ is simple, then any projection $p\in D$ has a subprojection in $D$ that uniformly approximately halves $p$ in trace, a property which we call \emph{tracial almost divisibility} (\Cref{def:tracialalmostdivisibility}), inspired by \cite[Definition~2.5]{Win12}.

\begin{theorem}\label{intro thm tad}
Let $(D \subset A)$ be a unital, ample diagonal pair with $A$ simple and infinite-dimensional and such that $\mathrm{T}(A)$ is nonempty. Then, $(D \subset A)$ has tracial almost divisibility.
\end{theorem}

Tracial almost divisibility of $(D \subset A)$ is what grants us access to the interpolation arguments appearing in the proof of \Cref{intro main thm}, since we can approximate arbitrarily well any affine continuous map over the trace space $\mathrm{T}(A)$ with values in the interval $[0,1]$ by a map that arises as evaluation of traces at a projection in $D$ (see \Cref{cor:evprojdense}).

The assumption of a hereditary conditional expectation in \Cref{intro main thm} is somewhat curious. Clearly this is a strengthening of faithfulness of $\Phi$, and it is actually automatically satisfied for conditional expectations associated to $\mathrm{C}^\ast$-pairs arising from amenable transformation groupoids (\Cref{rmk:herexponcrossedprod}) even without any assumptions on isotropy, but the general situation is rather obscure. We explore this further in \Cref{sec:herexp}, where in \Cref{ex:nonhereditary} we give an example of a (non-regular) $\mathrm{C}^*$-pair where the associated faithful expectation fails to be hereditary. Contrary to this instance, we show that finiteness of diagonal dimension of a $\mathrm{C}^*$-pair implies that the (uniquely determined) associated faithful expectation is hereditary.

\begin{proposition}\label{intro prop finite diag dim her exp}
Let $(D \subset A)$ be a unital diagonal pair with associated conditional expectation $\Phi\colon A \to D$. If $\dim_\mathrm{diag}(D \subset A) < \infty$, then $\Phi$ is hereditary. 
\end{proposition}

It is plausible that the conditional expectation of a diagonal pair $(D \subset A)$ is automatically hereditary when $A$ is nuclear, or even just exact.

\subsection*{Acknowledgements} This work was funded by the Deutsche Forschungsgemeinschaft
(DFG, German Research Foundation) under Germany's Excellence Strategy EXC 2044-390685587,
Mathematics M{\"u}nster: Dynamics--Geometry--Structure, by the SFB 1442 of the DFG, and by ERC Advanced Grant 834267 -- AMAREC.

\section{Diagonal comparison}
\renewcommand{\thetheorem}{\thesection.\arabic{theorem}}
\setcounter{theorem}{0}

\noindent We begin this section by establishing some notation and preliminaries.

For a $\rm{C}^\ast$-algebra $A$, we write $A_+$ to denote the cone of positive elements; $A^1$ will denote the closed unit ball of $A$. The set of positive contractions in $A$ is denoted by $A_+^1$. For $a \in A_+$, we denote by $\her(a)$ the hereditary sub-$\mathrm{C}^\ast$-algebra of $A$ generated by $a$. We will refer to $A$ together with a specified sub-$\mathrm{C}^\ast$-algebra $D \subset A$ as the $\mathrm{C}^\ast$-\emph{pair} $(D\subset A)$.

For $n\in\mathbb{N}$, we write $(D_n\subset M_n)$ for the pair of diagonal $n\times n$ matrices in $n\times n$ matrices and $E_n\colon M_n \to D_n$ denotes the unique faithful conditional expectation. For a $\rm{C}^\ast$-algebra $A$, we will naturally identify the trace space of the matrix amplification $A\otimes M_n$ with the trace space of $A$. We recall the definition of a Cartan pair and a $\mathrm{C}^\ast$-diagonal from \cite{Ren08, Kum86}.

\begin{definition}
Let $A$ be a $\mathrm{C}^\ast$-algebra and $D \subset A$ a sub-$\mathrm{C}^\ast$-algebra. We say that $(D \subset A)$ is a \emph{Cartan pair} (or that $D$ is a \emph{Cartan sub}-$\mathrm{C}^\ast$-\emph{algebra} of $A$) when the following conditions are satisfied:
\begin{enumerate}[label=(\arabic*)]
\setcounter{enumi}{-1}
\item\label{itcartandef2} $D$ contains an approximate unit of $A$ (non-degeneracy),
\item\label{itcartandef1} $D$ is a maximal abelian $^\ast$-subalgebra of $A$ (masa),
\item\label{itcartandef3} the set of \emph{normalisers} of $(D \subset A)$, namely
\[
\mc{N}_A(D)\coloneqq \{ v \in A \colon vDv^* \cup v^*D v \subset D \},
\]
generates $A$ as a $\mathrm{C}^\ast$-algebra (regularity),
\item\label{itcartandef4} there exists a faithful conditional expectation $\Phi\colon A \to D$.
\end{enumerate}
If moreover $(D \subset A)$ has the \emph{unique extension property}, i.e.\ every pure state on $D$ extends uniquely to a pure state on $A$, we will say that $(D \subset A)$ is a diagonal pair (or that $D$ is a $\mathrm{C}^\ast$-\emph{diagonal} of $A$).
\end{definition}

Recall that the conditional expectation in \ref{itcartandef4} is unique \cite[Corollary~5.7]{Ren08}, so we refer to it as \emph{the associated} conditional expectation of the Cartan pair. Moreover, it was recently shown in \cite{Pit21} that \ref{itcartandef1} and \ref{itcartandef3} make condition \ref{itcartandef2} redundant.

Our focus in this work is on diagonal $\mathrm{C}^\ast$-pairs $(D\sub A)$ where $D$ has zero-dimensional spectrum. Such diagonal $\mathrm{C}^*$-pairs are called \textit{ample}. For an element $d\in D$ we will write $\supp(d)$ to denote the \textit{open support} of $d$, i.e.\ the set of points in the spectrum of $D$ on which $d$ does not vanish.  

For a $\mathrm{C}^*$-algebra $A$ and $a,b\in A_+$, we say that $a$ is \textit{Cuntz subequivalent} to $b$ and write $a\precsim b$ when there exists a sequence $(r_n)_{n\ge1}\sub A$ such that $r_nbr_n^*\to a$. Letting $\mathrm{QT}(A)$ denote the space of normalised quasitraces on $A$ \cite[Definition~II.6.8.15]{Bla06}, for $\tau\in\mathrm{QT}(A)$ we define the \textit{dimension function} associated to $\tau$ as the map $d_\tau\colon\bigcup_{n=1}^\infty M_n(A)_+\to[0,\infty)$ given by
\begin{equation*}
d_\tau(a):=\lim_{k\to\infty}\tau(a^{1/k}), \quad a\in\bigcup_{n=1}^\infty M_n(A)_+.
\end{equation*}
Clearly, for a projection $p\in\bigcup_{n=1}^\infty M_n(A)_+$ one has $d_\tau(p)=\tau(p)$. For $a,b\in A_+$, if $a\precsim b$ then $d_\tau(a)\le d_\tau(b)$ and if $a\perp b$ then $d_\tau(a+b)=d_\tau(a)+d_\tau(b)$ \cite{BlaHan82}. 

\begin{definition}{\cite[Definition~1.5]{BosBroSatTikWhiWin19}}\label{def:strictcomp}
For a unital, separable, and simple $\mathrm{C}^*$-algebra $A$ such that $\mathrm{QT}(A)$ is nonempty, we say that $A$ has \textit{strict comparison} when, for all $n\in\bb{N}$ and all $a,b\in M_n(A)_+$, if $d_\tau(a)<d_\tau(b)$ for all $\tau\in\mathrm{QT}(A)$, then $a\precsim b$ in $M_n(A)_+$.
\end{definition}

\begin{remark}\label{rmk:dynamicalsubequivalencenormaliser}
The notion of dynamical subequivalence in topological dynamics \cite[Definition~3.1]{Ker20} admits a straightforward generalisation to ample $\mathrm{C}^\ast$-pairs. To motivate this, note that if $G\acts X$ is a continuous action on a zero-dimensional space and $U, V \subset X$ are clopen sets such that $U$ is partitioned as $U=\bigsqcup_{j=1}^nU_j$ and there are group elements $\{s_j\}_{j=1}^n\subset G$ such that $\bigsqcup_{j=1}^ns_jU_j\subset V$, then the element $w\coloneqq\sum_{j=1}^n\chi_{U_j}u_{s_j}^* \in C(X)\rtimes_\mathrm{r}G$ is a normaliser of $(C(X) \subset C(X)\rtimes_\mathrm{r}G)$ and satisfies $w\chi_Vw^*=\chi_U$. 
 
More generally, for a locally compact, {\'e}tale, Hausdorff, ample groupoid $\mathcal{G}$ with compact unit space $\mathcal{G}^{(0)}$ (see \cite{Sims:Notes} for the details of this setting), if $U, V$ are clopen subsets of the unit space, $U$ is dynamically below $V$ as in \cite[Definition~2.1]{AraBonBosLi23} when there are compact open bisections $W_1 ,\dots, W_n \subset \mathcal{G}$ with $U = \bigsqcup_{j=1}^n s(W_j)$ and $\bigsqcup_{j=1}^nr(W_j) \subset V$, where $s, r \colon \mathcal{G} \to \mathcal{G}^{(0)}$ denote the source and range maps respectively. In this case, the compactly supported function $\mathcal{G} \to \mathbb{C}$ given by $w\coloneqq \sum_{j=1}^n \chi_{W_j}\in C_c(\mathcal{G}) \subset \mathrm{C}^\ast_\mathrm{r}(\mathcal{G})$ is a normaliser of $(C(\mathcal{G}^{(0)}) \subset \mathrm{C}^\ast_\mathrm{r}(\mathcal{G}))$ and satisfies $w^*\chi_V w=\chi_U$. 

These observations naturally lead to the following definition.
\end{remark}

\begin{definition}\label{def:dynbelow}
Let $(D\subset A)$ be a unital, ample $\mathrm{C}^\ast$-pair and let $p,q\in D$ be projections. We say that $p$ is \emph{dynamically below} $q$ and write $p\prec_{(D \sub A)} q$ when there is a normaliser $n\in\mc{N}_A(D)$ such that $p=nqn^*$. Also, we will say that $p$ and $q$ are \emph{dynamically equivalent} and write $p\sim_{(D \sub A)} q$ when $p$ and $q$ are Murray--von Neumann equivalent via a partial isometry in $\mc{N}_A(D)$. 
\end{definition}

Since for a unital diagonal pair $(D \subset A)$ the simplex of \emph{invariant} states (i.e.\ states $\sigma\in \mathrm{S}(D)$ such that $\sigma(n^*n)=\sigma(nn^*)$ for all $n\in\mathcal{N}_A(D)$) is affinely homeomorphic to the trace space $\mathrm{T}(A)$ \cite[Corollary~3.6]{CryNag17}, the following definition is a natural adaptation of dynamical comparison \cite[Definition~3.2, Proposition~3.5]{Ker20} for ample diagonal pairs.

\begin{definition}\label{def:dyncomp}
Let $(D\subset A)$ be a unital, ample diagonal pair with $\mathrm{T}(A)$ nonempty. We say that $(D\subset A)$ has \emph{dynamical comparison} when, for any projections $p,q \in D$, if $\tau(p) < \tau(q)$ for all $\tau\in\mathrm{T}(A)$, then $p \prec_{(D \sub A)} q$.
\end{definition}

\begin{remark}\label{rmk:stabledynamicalcomparison}
In contrast to strict comparison, \Cref{def:dyncomp} only addresses projections in $D$, as opposed to $\bigcup_{n\ge1}(D\otimes D_n)$. Nevertheless, at least when $A$ is assumed to be separable, zero-dimensionality of the spectrum of $D$ automatically implies that for any $n\in\bb{N}$, the pair $(D\otimes D_n \subset A\otimes M_n)$ has dynamical comparison as well. This is not trivial, but the proof is a straightforward adaptation of the argument in the proof of \cite[Proposition~2.10]{AraBonBosLi23} where the analogous statement is proved for the case of pairs arising as $(C(\mc{G}^{(0)})\subset \mathrm{C}^*_\mathrm{r}(\mathcal{G}))$ for a $\sigma$-compact, ample groupoid $\mathcal{G}$ with compact unit space $\mathcal{G}^{(0)}$, so we omit the details.
\end{remark}

In order to introduce the main new concept of this paper, diagonal comparison, we need a version of a comparison relation suited for diagonal pairs.

\begin{definition}\label{def:diagbelow}
Let $(D\subset A)$ be a $\rm{C}^\ast$-diagonal pair and let $\Phi\colon A\to D$ denote the associated conditional expectation. For $a,b\in A_+$, we say that $a$ is \emph{diagonally below} $b$ and write $a\precsimd b$ when for any $\varepsilon>0$ there are elements $u,w,v\in A$ such that
\begin{enumerate}
\item\label{it:defdiagbelow1} $\|a - u w v b v^* w^* u^* \|<\varepsilon$,
\item\label{it:defdiagbelow2} $\supp (\Phi(u^*u)) \subset \supp (\Phi(a))$,
\item\label{it:defdiagbelow3} $\supp (\Phi(vv^*)) \subset \supp (\Phi(b))$,
\item\label{it:defdiagbelow4} $w\in \mc{N}_A(D)$.
\end{enumerate}
\end{definition}

Notice that $\precsimd$ is stronger than the usual Cuntz subequivalence, since for $a \precsimd b$ to hold we are asking that $a$ is Cuntz below $b$ via elements that factorise as a product of three elements as described above. 
%Also notice that the relation $\precsimd$ is not designed to be transitive.

\begin{definition}\label{def:diagcomp}
Let $(D \subset A)$	be a unital, ample diagonal pair with $A$ separable, simple and such that $\mathrm{QT}(A)$ is nonempty. We say that $(D \subset A)$ has \emph{diagonal comparison} when for any $n\in\bb{N}$ and any $a,b\in (A\otimes M_n)_+$, if $d_\tau(a) < d_\tau(b)$ for all $\tau \in \mathrm{QT}(A)$, then $a\precsimd b$ in $(D\otimes D_n \subset A \otimes M_n)$. 
\end{definition}

Clearly diagonal comparison of a pair automatically implies strict comparison of the ambient $\rm{C}^\ast$-algebra. We will see that diagonal comparison implies that the pair has dynamical comparison, at least if one assumes that the conditional expectation is \emph{hereditary} in the following sense.

\begin{definition}\label{def:hermap}
Let $(D \subset A)$ be a $\rm{C}^\ast$-pair and let $\Phi\colon A\to D$ be a positive linear map. We say that $\Phi$ is \emph{hereditary} when $a\in \her(\Phi(a))$ for all $a\in A_+$.
\end{definition}

It is obvious that if a positive map is hereditary then it is automatically faithful. This strengthening of faithfulness occurs for the conditional expectations associated to diagonal pairs in a fairly broad range of examples.

\begin{remark}\label{rmk:herexponcrossedprod}
For a continuous action $G\acts X$ of an amenable discrete group on a compact Hausdorff space $X$, the canonical conditional expectation $E\colon C(X)\rtimes_\mathrm{r}G\to C(X)$ is hereditary (cf.\ \cite[Remark~2.3]{RorSie12}). The reason for this is that a F{\o}lner sequence $(F_n)_{n=1}^\infty$ for $G$ with $|F_n|\eqqcolon k_n\in\mathbb{N}$ gives rise to a completely positive approximation property of the form

\[
\begin{tikzcd}
C(X)\rtimes_\mathrm{r}G\ar[ "\psi_n", dr] & & C(X)\rtimes_\mathrm{r}G \\
& C(X) \otimes M_{k_n}\ar["\varphi_n", ur] &
\end{tikzcd}
\]

\noindent with $\varphi_n\psi_n(fu_s) = \frac{|F_n \cap sF_n|}{|F_n|} fu_s$ for all $f\in C(X)$, $s\in G$ \cite[Lemma~4.2.3]{BroOza08}. This composition is thus expressed in terms of the \emph{Fourier coefficients}, i.e.\ the completely bounded maps $E_s\colon C(X)\rtimes_\mathrm{r}G\to C(X)$, $s\in G$, given by $E_s(\cdot)\coloneqq E(\cdot \; u_s^*)$,  since $\varphi_n\psi_n(\cdot) = \sum_{s\in G}\frac{|F_n\cap sF_n|}{|F_n|}E_s(\cdot)u_s$.\footnote{The sum appearing in this equation is actually finite: since $F_n$ is finite, the intersection $F_n\cap sF_n$ can only be nonempty for finitely many $s\in G$.} Now for $a\in (C(X)\rtimes_\mathrm{r}G)_+$, we have that
\begin{equation}\label{eq:rmk}
a = \lim_{n\to\infty} \varphi_n\psi_n(a) = \lim_{n\to \infty} \sum_{s\in G}\frac{|F_n\cap sF_n|}{|F_n|}E_s(a)u_s,
\end{equation}
and so, using the Cauchy--Schwarz inequality for completely positive maps on $E$, it is easy to see that if $a_s\coloneqq E_s(a)u_s$, then $a_s^*a_s\in \her(E(a))$ and $a_sa_s^*\in\her(E(a))$, whence $a_s\in \her(E(a))$. By \eqref{eq:rmk}, $a\in\her(E(a))$ as well, and this shows that $E$ is hereditary.

The same argument can be used to show that the conditional expectation $\mathrm{C}^\ast_\mathrm{red}(\mc{B})\to B_{1_G}$ associated to the $\mathrm{C}^\ast$-algebra of a Fell bundle $\mc{B}=\{B_g\}_{g\in G}$ over an amenable group with commutative unit fiber $B_{1_G}$ and with the approximation property (see \cite[Part~2]{Exe17} for the details of this setting) is hereditary.

\end{remark}

It is an interesting question to what extent having a hereditary conditional expectation is automatic -- this topic is further discussed in \Cref{sec:herexp}. 

\begin{lemma}\label{lem:positive-map-her-inclusion}
Let $ (D \subset A)$ be a $\rm{C}^\ast$-pair with $D$ abelian and let $\Phi \colon A \to D$ be a positive linear map. For $a\in A_+$, we have that $\Phi(\her(a)) \subset \her(\Phi(a))$. 
\end{lemma}

\begin{proof}
Let $b\in\her(a)_+$. Since $\her(a)$ is a right $\mathrm{C}^\ast(a)$-module, by the Cohen--Hewitt factorisation theorem there are $a_0\in\mathrm{C}^\ast(a)$ and $c\in A$ such that $b^{1/2}=ca_0$ and therefore $0\le \Phi(b) \le \|c\|^2 \Phi( a_0^*a_0)$, so without loss of generality we may assume that $b \in \mathrm{C}^\ast(a)_+$. Since $\Phi$ is continuous and $b$ is positive, it suffices to prove that $b\in\her(\Phi(a))$ assuming that $b=f(a)$ where $f\in C_0(\sigma(a))$ is the positive part of a polynomial with zero constant term. For such a function $f$, there is a constant $\lambda>0$ such that $f(t)\le \lambda t$ for all $t\in\sigma(a)$, whence $0 \le b \le \lambda a$ and therefore positivity of $\Phi$ implies that $\Phi(b)\in\her(\Phi(a))$ as we wanted.
\end{proof}

\begin{lemma}\label{lem:herdiagbelow}
Let $(D \subset A)$ be a diagonal pair with hereditary associated conditional expectation $\Phi\colon A\to D$ and let $a,b\in A_+$. Then, $a\precsimd b$ if and only if for any $\varepsilon>0$ there exist $u,w,v\in A$ such that
\begin{enumerate}[label=\normalfont(\roman*)]
\item\label{it:lemherdbelow1} $\|a - u w v b v^* w^* u^*\|<\varepsilon$,
\item\label{it:lemherdbelow2} $u \in \her(\Phi(a))$,
\item\label{it:lemherdbelow3} $v \in \her(\Phi(b))$,
\item\label{it:lemherdbelow4} $w \in \mc{N}_A(D)$.
\end{enumerate}
\end{lemma}

\begin{proof}
In view of \Cref{lem:positive-map-her-inclusion}, the reverse implication is immediate. Assume now that $a \precsimd b$ and let $\varepsilon>0$. There are $u_0,w,v_0\in A$ with $\supp(\Phi(u_0^*u_0))\sub\supp(\Phi(a))$, $\supp(\Phi(v_0v_0^*))\sub\supp(\Phi(b))$ and $w\in\mc{N}_A(D)$ such that $\|a-u_0wv_0bv_0^*w^*u_0^*\|<\varepsilon/2$. Since $\lim_{m\to\infty}c^{1+\frac{1}{m}}=c$ for all $c\in A_+$, we can take $m\in\bb{N}$ large enough so that
\begin{equation*}
\|a-a^\frac{1}{m}u_0wv_0b^{1+\frac{2}{m}}v_0^*w^*u_0^*a^\frac{1}{m}\|<\varepsilon.
\end{equation*}
Set $u:=a^\frac{1}{m}u_0$ and $v:=v_0b^\frac{1}{m}$. It is clear that the elements $u,w,v$ satisfy properties \ref{it:lemherdbelow1} and \ref{it:lemherdbelow4}. Invoking that $\Phi$ is hereditary and by another application of \Cref{lem:positive-map-her-inclusion}, we see that $u^*u, uu^* \in \her(\Phi(a))$ and $vv^*, v^*v \in \her(\Phi(b))$, and so \ref{it:lemherdbelow2} and \ref{it:lemherdbelow3} follow.
\end{proof}

\begin{proposition}\label{prop:diagcomptodyncomp}
Let $(D\subset A)$ be a unital, ample diagonal pair with $A$ separable, simple and such that $\mathrm{QT}(A)=\mathrm{T}(A)$ is nonempty. Assume that the associated conditional expectation is hereditary. If $(D\subset A)$ has diagonal comparison, then $(D\subset A)$ has dynamical comparison.
\end{proposition}

\begin{proof}
Let $p,q \in D$ be projections with $\tau(p) < \tau(q)$ for all $\tau\in\mathrm{T}(A)$. If $p = 0$ there is nothing to show, so we assume that $p \ne 0$. By diagonal comparison we have $p \precsimd q$, so by \Cref{lem:herdiagbelow} there are $u\in\her(p)$, $v\in\her(q)$ and $w\in\mc{N}_A(D)$ such that 
\begin{equation*}
\|p - u w v q v^* w^* u^* \| < 1,
\end{equation*}
In particular, $u, w, v$ are non-zero, and $u w v q v^* w^* u^* \in \her(p)$ is a positive element that is invertible in $\her(p)$. Moreover, since $v \in \her(q)$, we have $vqv^* \le \|v\|^2 q$, and so $ u w q w^* u^* $ is invertible in $\her(p)$, and since $0\le uwqw^*u^* \le \|w\|^2 uu^*$, $uu^*$ is also invertible in $\her(p)$. Since $\her(p)$ is stably finite as a hereditary subalgebra of $A$, right invertibility of $u$ implies that $u$ itself is invertible in $\her(p)$, and thus so is $pwqw^*p$. Now $D$ is abelian and by invertibility of $pwqw^*p$ in $\her(p)$ it follows that $wqw^*\in D$ does not vanish on the support of $p$ which is a clopen set, whence there is $h\in D$ such that $hwqw^*h=p$. Setting $\tilde{w}:=hw\in\mc{N}_A(D)$ we have that $p=\tilde{w}q\tilde{w}^*$, proving that $p \prec_{(D \sub A)} q$ as we wanted.
\end{proof}

Note that at this point we have already shown the implication \ref{itm2intromainthm}$\implies$\ref{itm1intromainthm} of \Cref{intro main thm}.

\begin{remark}\label{rmk:transitivity}
We note that the relation $\precsimd$ is not designed to be a priori
transitive. However, in the presence of diagonal comparison, the combination of $\precsimd$ with strict subequality
on all dimension functions, indeed yields a transitive relation. It remains an interesting technical problem whether the
relation $\precsimd$ alone can fail to be transitive.
\end{remark}

\section{Tracial almost divisibility}
\noindent In order to establish that the concurrence of dynamical comparison of an ample diagonal pair $(D \subset A)$ together with strict comparison of $A$ imply that $(D\subset A)$ has diagonal comparison (i.e.\ \ref{itm1intromainthm}$\implies$\ref{itm2intromainthm} of \Cref{intro main thm}), it will be necessary that we are able to partition any given projection $p \in D$ into subprojections that, uniformly over $\mathrm{T}(A)$, approximately preserve in trace a prescribed portion of the trace of $p$. 

We will use $\mathrm{Aff}(\mathrm{T}(A))$ to denote the linear space of weak$^\ast$-continuous affine maps $\mathrm{T}(A) \to \bb{C}$ and $\mathrm{Aff}(\mathrm{T}(A))_+^1$ will denote the subset of those maps with values in $[0,1]$. For $f,g\in\mathrm{Aff}(\mathrm{T}(A))$ and $\varepsilon>0$, we write $f\approx_\varepsilon g$ when $|f(\tau)-g(\tau)|<\varepsilon$ for all $\tau\in\mathrm{T}(A)$. For an element $a\in A_+$, we use $\hat{a}$ to denote the affine continuous map $\mathrm{T}(A)\ni\tau\mapsto\tau(a)\in[0,\|a\|]$.

\begin{definition}\label{def:tracialalmostdivisibility}
Let $(D \subset A)$ be a unital, ample diagonal pair with $\mathrm{T}(A)$ nonempty. We say that $(D \subset A)$ has \emph{tracial almost divisibility} when for any projection $p \in D$, any $n \in \bb{N}$ and any $\varepsilon>0$ there exist projections $p_1, \dots ,p_n \in D$ such that
\begin{enumerate}
\item $\sum_{j=1}^n p_j = p$, 
\item $|\tau(p_j) - \frac{1}{n} \tau (p) | < \varepsilon$ for all $\tau\in \mathrm{T}(A)$, $j=1,\dots,n$.
\end{enumerate}
\end{definition}

\begin{lemma}\label{lem:tadequivalents}
Let $(D \subset A)$ be a unital, ample diagonal pair with $\mathrm{T}(A)$ nonempty. The following are equivalent:
\begin{enumerate}[label=\normalfont(\roman*)]
\item\label{it:tadequiv1} $(D \subset A)$ has tracial almost divisibility.
\item\label{it:tadequiv2} For any projection $p \in D$ and any $\varepsilon>0$ there exists $p'\in D$, a subprojection of $p$, such that $|\tau(p') - \frac12\tau(p)|<\varepsilon$ for all $\tau \in \mathrm{T}(A)$.
\item\label{it:tadequiv3} For any projection $p \in D$, any $\lambda\in[0,1]$ and any $\varepsilon>0$ there exists $p'\in D$, a subprojection of $p$, such that $|\tau(p') - \lambda \tau(p)|<\varepsilon$ for all $\tau \in \mathrm{T}(A)$.
\end{enumerate}
\end{lemma}
\begin{proof}
\ref{it:tadequiv1}$\implies$\ref{it:tadequiv2} is trivial. 

\ref{it:tadequiv2}$\implies$\ref{it:tadequiv3}. Let $p\in D$ be a projection and let $\lambda\in[0,1]$, $\varepsilon>0$ be given. Find $k \in \bb{N}$ and $1 \le m \le 2^k$ such that $|\lambda - \frac{m}{2^k}|<\varepsilon/3$. Set $\delta\coloneqq \varepsilon/(3m)$. By \ref{it:tadequiv2}, there is $p_0\in D$ a subprojection of $p$ such that $\hat{p}_0\approx_\delta\frac{1}{2}\hat{p}$ and note that $p_1 \coloneqq p-p_0$ also satisfies $\hat{p}_1 \approx_\delta \frac{1}{2}\hat{p}$ and $p_0+p_1=p$. By iterated applications of \ref{it:tadequiv2}, each time on the subprojections obtained in the previous step, after $k$ steps we obtain $2^k$ projections $q_1,\dots, q_{2^k} \in D$ such that $\sum_{j=1}^{2^k}q_j=p$ and $\hat{q}_j\approx_\eta\frac{1}{2^k}\hat{p}$, where $\eta\coloneqq \sum_{j=0}^{k-1}\frac{1}{2^j}\delta < 2\delta$. Set $p'\coloneqq \sum_{j=1}^mq_j$. Then $\hat{p}' = \sum_{j=1}^m\hat{q}_j \approx_{m\eta} \frac{m}{2^k}\hat{p}\approx_{\varepsilon/3}\lambda\hat{p}$ and since $m\eta +\varepsilon/3 <\varepsilon$, we obtain $\hat{p}' \approx_\varepsilon\lambda\hat{p}$.

\ref{it:tadequiv3}$\implies$\ref{it:tadequiv1}. Let $\varepsilon>0$, let $n\in\bb{N}$ and let $p\in D$ be a projection. Set $\delta\coloneqq 2^{-n+1}\varepsilon$. By \ref{it:tadequiv3}, there is $p_1\in D$ a subprojection of $p$ such that $\hat{p}_1 \approx_\delta \frac1n\hat{p}$. Again by \ref{it:tadequiv3}, $p-p_1$ has a subprojection $p_2 \in D$ such that $\hat{p}_2 \approx_\delta \frac{1}{n-1}\widehat{p - p_1}$. Since $\widehat{p-p_1}\approx_\delta\frac{n-1}{n}\hat{p}$, we have $\hat{p}_2\approx_{2\delta}\frac{1}{n}\hat{p}$. Continuing in this fashion, we obtain $p_1,\dots,p_{n-1}\in D$ that are mutually orthogonal subprojections of $p$ and such that $\hat{p}_j\approx_{2^{j-1}\delta}\frac{1}{n}\hat{p}$. Setting $p_n\coloneqq p - \sum_{j=1}^{n-1}p_j$, we have $\hat{p}_n\approx_{(2^{n-1}-1)\delta}\frac{1}{n}\hat{p}$, and so $\hat{p}_j\approx_\varepsilon\hat{p}$ for all $j=1,\dots,n$ as we wanted.
\end{proof}

\begin{proposition}\label{prop:orderedprojs}
Let $(D \subset A)$ be a unital, ample diagonal pair with $A$ simple, and let $p, q\in D$ be projections with $q\ne0$. Then, there exist projections $p_1,\dots,p_n\in D$ such that $p=\sum_{j=1}^np_j$ and $p_j \prec_{(D \sub A)} q$ for all $j=1,\dots,n$. Moreover, we can arrange that $p_1 \prec_{(D \sub A)} \dots \prec_{(D \sub A)} p_n \prec_{(D \sub A)} q$.
\end{proposition}

\begin{proof}
For the first part of the statement, consider $J\coloneqq\overline{\mathrm{span}}\{vqv^* : v\in \mc{N}_A(D)\} \trianglelefteq D$, and note that $vJv^* \subset J$ for all $v \in \mc{N}_A(D)$. Set $I\coloneqq \overline{\mathrm{span}}\{ fv: f\in J, v\in \mc{N}_A(D)\}$ which is a closed linear subspace of $A$, and since $(D \subset A)$ is regular, it is readily seen that $I$ is a left ideal in $A$. Actually, $I$ is self-adjoint. Indeed, let $f\in J$ and $v\in\mc{N}_A(D)$. We have
\[
(fv)^* = v^*f^* = \lim_{n\to\infty} v^*\underbrace{(vv^*)^\frac{1}{n}}_{\in D} f^* = \lim_{n\to\infty} v^* f^* (vv^*)^\frac{1}{n},
\]
so it suffices to see that $v^*f^*(vv^*)^\frac{1}{n} \in I$ for all $n\in\mathbb{N}$. By Weierstrass's approximation theorem, it suffices to see that $v^*f^*(vv^*)^k \in I$ for all $k\in\mathbb{N}$, which is immediate, since $v^*f^* (vv^*)^k = (v^*f^*v) \cdot v^* (vv^*)^{k-1}$, and $v^* f^* v \in J$, so $(v^* f^* v) v^* \in I$ and thus multiplying $(v^*f^*v)v^*$ on the right with $(vv^*)^{k-1}$ yields again an element of $I$. Therefore, $I$ is a non-trivial ideal in $A$, and by simplicity it follows that $I=A$. If $\Phi\colon A\to D$ is the associated conditional expectation, since $\Phi(I) = J$, we have that $J=D$, i.e.\ $D = \overline{\mathrm{span}}\{ vqv^* : v \in \mc{N}_A(D)\}$. We can thus find $v_1,\dots,v_m\in\mc{N}_A(D)$ such that $1= \sum_{i=1}^m v_iqv_i^*$, and therefore $\mathcal{U}\coloneqq \{\supp(v_iqv_i^*)\}_{i=1}^m$ is an open cover of the spectrum of $D$, and by zero-dimensionality we obtain a partition of the spectrum of $D$ that refines $\mathcal{U}$. The indicator functions of the clopen sets in this partition thus give us projections $p_1',\dots, p_n' \in D$ such that $1=\sum_{j=1}^np_j'$, and for each $j=1,\dots,n$ there is some $1\le i \le m$ such that $\supp(p_j') \subset \supp( v_iqv_i^*)$, so in particular (by Tietze's extension theorem) there is $h\in D_+$ such that $p_j' = h^{1/2} v_i q v_i^* h^{1/2}$, which is to say that $p_j' \prec_{(D \sub A)} q$.  Setting $p_j:=pp_j'$ for each $j=1, \dots ,n$ finishes the proof of the first part of the statement.

For the second claim, write $p=\sum_{j=1}^np_j$ with $p_j \in D$ projections such that $p_j \prec_{(D \subset A)} q$ for all $j=1,\dots,n$. Let $t_j\in\mc{N}_A(D)$ be a partial isometry with $p_j = t_jt_j^*$ and  $q_j\coloneqq t_j^*t_j \le q$. For $S\subset \{1,\dots,n\}$, set
\[
q_S \coloneqq \prod_{j\in S}q_j\cdot\prod_{j\not\in S}(q-q_j).
\]
which is a projection, since $D$ is abelian. It is clear that if $S, T \subset \{1,\dots,n\}$ are distinct, then $q_S \perp q_T$. Since for any $j=1,\dots,n$ we have $q_j = \sum_{S \ni j}q_S$, we also have that $p_j = t_j q_j t_j^* = \sum_{S\ni j} t_jq_St_j^*$, whence
\begin{equation}\label{eq: p sum}
p = \sum_{j=1}^n \sum_{S \ni j} t_j q_St_j^*.
\end{equation}
For $j= 1,\dots,n$, and $S \subset \{1,\dots,n\}$ with $|S| \ge n+1-j$, let $k(j,S)\in S$ be the $(n+1-j)$th element of $S$ and set
\begin{equation}\label{eq: ordered projections}
\tilde{p}_j\coloneqq \sum_{|S| \ge n+1 -j} t_{k(j,S)} q_S t_{k(j,S)}^*.
\end{equation}
Note that each $\tilde{p}_j$ is a subprojection of $p$: indeed, let $S,T \subset \{1,\dots,n\}$ be two distinct subsets with cardinality at least $n+1-j$. If $k(j,S)=k(j,T)$, then $t_{k(j,S)}q_St_{k(j,S)}^* t_{k(j,T)}q_Tt_{k(j,T)}^* = t_{k(j,S)} ( q_S q_T ) t_{k(j,S)}^* = 0$. If $k(j,S) \ne k(j,T)$, then $t_{k(j,S)}^*t_{k(j,T)}=0$. In any case, $t_{k(j,S)}q_St_{k(j,S)}^* \perp t_{k(j,T)}q_Tt_{k(j,T)}^*$, and so $\tilde{p}_j$ is a sum of pairwise orthogonal subprojections of $p$, and so it is itself a subprojection of $p$. For $j=1,\dots,n$, we also set
\[
v_j \coloneqq \sum_{|S| \ge n+1-j} t_{k(j,S)}q_S.
\]
With the same reasoning that we used to see that the terms appearing in the sum in \eqref{eq: ordered projections} are pairwise orthogonal, we see that, whenever $S,T$ are  distinct, we have $(t_{k(j,S)}q_S)^*(t_{k(j,T)}q_T) = 0$ and also $(t_{k(j,S)} q_S) (t_{k(j,T)}q_T)^* = 0$, so \cite[Lemma~1.4]{LiaTik22} yields that $v_j \in \mc{N}_A(D)$. Clearly $v_jv_j^*=\tilde{p}_j$. On the other hand, $v_j^*v_j = \sum_{|S| \ge n+1-j}q_S$, therefore $v_1^*v_1 \le v_2^* v_2 \le \dots \le v^*_nv_n \le q$, which is to say that $\tilde{p}_1 \prec_{(D \sub A)} \tilde{p}_2 \prec_{(D \sub A)} \dots \prec \tilde{p}_n \prec_{(D \sub A)} q$. 

We claim that $\sum_{j=1}^n\tilde{p}_j = p$. Let $j\in\{1,\dots,n\}$, and let $S \subset \{1,\dots,n\}$ with $|S| \ge n+1-j$. Then, $t_{k(j,S)}q_St_{k(j,S)}^* \le t_{k(j,S)}t_{k(j,S)}^*=p_{k(j,S)}\le p$, and so by \eqref{eq: ordered projections} and the fact that the projections appearing in the sum therein are pairwise orthogonal, we see that 
\begin{equation}\label{eq:auxiliaryedit}
\tilde{p}_j\le p, \quad j=1,\dots,n.
\end{equation}
Moreover, the projections $\{\tilde{p}_j\}_{j=1}^n$ are pairwise orthogonal. Indeed, let $1 \le i,j \le n$ be distinct and $S,T \subset \{1,\dots,n\}$ with $|S|\ge n+1-j$, $|T| \ge n+1 -i$. If $k(j,S) \ne k(i,T)$, then $t_{k(j,S)}q_St_{k(j,S)}^* \perp t_{k(i,T)}q_Tt_{k(i,T)}^*$ If $k(j,S)=k(i,T)$, then $S$ and $T$ cannot coincide, since if they were the same set we would have $i=j$. Therefore, $t_{k(j,S)}q_St_{k(j,S)}^*t_{k(i,T)}q_Tt_{k(i,T)}^* = t_{k(j,S)}(q_Sq_T)t_{k(j,S)}^*=0$. By \eqref{eq: ordered projections}, we conclude that $\tilde{p}_j\perp\tilde{p}_i$, and so by \eqref{eq:auxiliaryedit} we have that $\sum_{j=1}^n\tilde{p}_j\le p$. For the reverse inequality, let $1\le j \le n$ and $S \subset \{1,\dots,n\}$ with $j\in S$. Let $i \in \{1,\dots,n\}$ be the unique integer so that $j$ is the $(n+1-i)$th element of $S$, so $j=k(i,S)$ and thus $t_jq_St_j^* = t_{k(i,S)}q_St_{k(i,S)}^*$, so $t_jq_St_j^* \le \tilde{p}_i$. By \eqref{eq: p sum}, we conclude that $p\le\sum_{i=1}^n\tilde{p}_i$, and thus $p=\sum_{i=1}^n\tilde{p}_i$. This completes the proof.
\end{proof}

\begin{theorem}\label{thm:trdivsimple}
Let $(D \subset A)$ be a unital, ample diagonal pair with $A$ simple and infinite-dimensional and such that $\mathrm{T}(A)$ is nonempty. Then, $(D \subset A)$ has tracial almost divisibility.
\end{theorem}

\begin{proof}
Let $\varepsilon>0$ and let $p\in D$ be a projection. Let $\widehat{D}$ denote the spectrum of $D$. For $\tau\in\mathrm{T}(A)$ we have that $\tau\vert_D$ corresponds to a Borel probability measure $\mu_\tau$ on $\widehat{D}$ that is invariant for the partial action of the normalisers of the Cartan pair described in \cite[Proposition~4.6]{Ren08}. Since $A$ is simple, the orbit of any point in $\widehat{D}$ under the partial action of the normalisers needs to be dense in $\widehat{D}$ and in particular infinite, since $\widehat{D}$ is infinite (due to $A$ being infinite-dimensional), which is to say that $\mu_\tau$ is atomless. By zero-dimensionality of $\widehat{D}$, consider a decreasing sequence of projections $(q_n)_{n\ge1} \subset D$ with $\bigcap_{n\ge1}\supp(q_n)$ being a singleton. Then $\tau(q_n)\to0$ for any $\tau\in\mathrm{T}(A)$ and, since the continuous maps $\hat{q}_n\colon\mathrm{T}(A)\to[0,1]$ are pointwise decreasing to $0$, by Dini's theorem this convergence is uniform. By going far out in this sequence $(q_n)_{n\ge1}$, we can thus obtain an auxiliary non-zero projection $q\in D$ such that $\tau(q)<\varepsilon$ for all $\tau\in\mathrm{T}(A)$. By \Cref{prop:orderedprojs}, we write $p = \sum_{j=1}^np_j$ where each $p_j\in D$ is a projection and $p_1 \prec_{(D \sub A)} \dots \prec_{(D \sub A)} p_n \prec_{(D \sub A)} q$. Without loss of generality we can assume that $n$ is even (otherwise set $p_0\coloneqq0$ and consider the $n+1$ projections $p_0, \dots, p_n$ instead). Let $p' \coloneqq \sum_{2j \le n} p_{2j}$ be the sum of the even-numbered projections. For $\tau\in\mathrm{T}(A)$, we clearly have $\tau(p-p') \le \tau (p') $ and also $\tau(p_{2j})\le \tau(p_{2j+1})$ for all $j<n/2$, whence $\tau(p') \le \tau(p-p')+\tau(p_{2n})<\tau(p-p')+\varepsilon$. It follows that $\tau(p) \le 2\tau(p') < \tau(p)+\varepsilon$ and thus $|\tau(p') - \frac{1}{2}\tau(p)|<\varepsilon$ for all $\tau\in \mathrm{T}(A)$. The conclusion now follows from \Cref{lem:tadequivalents}.
\end{proof}

\begin{corollary}\label{cor:evprojdense}
Let $(D \subset A)$ be a unital, ample diagonal pair with $A$ simple and infinite-dimensional. Then, $\{\hat{p}: p \in D \mbox{ is a projection}\}$ is $\|\cdot\|_\infty$-dense in $\mathrm{Aff}(\mathrm{T}(A))_+^1$.
\end{corollary}

\begin{proof}
Let $f \in\mathrm{Aff}(\mathrm{T}(A))_+^1$ and let $\varepsilon>0$. We aim to approximate $f$ sufficiently well, so by \cite[Lemma~6.2]{KirRor14} we can assume without loss of generality that $f=\hat{a}$ for some $a\in A_+^1$. Since $(D \subset A)$ is a diagonal pair, all traces on $A$ factorise through the conditional expectation $\Phi\colon A\to D$ \cite[Lemma~4.3]{LiRen19}, \cite[Corollary~3.6]{CryNag17}, so upon replacing $a$ by $\Phi(a)$ we can further assume without harming the generality that $a\in D_+^1$. By compactness, zero-dimensionality of the spectrum of $D$ and uniform continuity of $a$ (as a continuous function over the spectrum of $D$), there are pairwise orthogonal projections $p_1, \dots, p_n \in D$ and $\lambda_1, \dots,\lambda_n \ge0$ such that $\| a - \sum_{i=1}^n\lambda_i p_i \|<\varepsilon$. By \Cref{thm:trdivsimple} and \ref{it:tadequiv3} of \Cref{lem:tadequivalents}, we obtain projections $p'_i\in D$ with $p'_i \le p_i$ and such that $\hat{p}'_i \approx_{\varepsilon/n} \lambda_i\hat{p}_i$ for $i=1,\dots,n$. Define $p':=\sum_{i=1}^np'_i\in D$ which is a projection and satisfies $\hat{p}' \approx_\varepsilon \sum_{i=1}^n \lambda_i \hat{p}_i$, whence $\hat{p}' \approx_{2\varepsilon} \hat{a}$.
\end{proof}

\section{Main result and consequences}

\noindent We begin with some technical lemmas that will be needed for the proof of the main theorem. For a self-adjoint element $a$ in a $\mathrm{C}^\ast$-algebra $A$, we will say that $a$ is \emph{locally invertible} when $a$ is invertible in $\her(a)$. Equivalently, $a$ is locally invertible if and only if $0$ is isolated in $\{0\}\cup \sigma(a)$, where $\sigma(a)$ denotes the spectrum of $a$. For $\varepsilon>0$, we will use $(a-\varepsilon)_+ \in \mathrm{C}^\ast(a)_+$ to denote the $\varepsilon$-cutdown of $a$, i.e.\ the element obtained by applying functional calculus on $a$ with the map $t\mapsto\max\{t-\varepsilon,0\}$.  Recall that for $a\in A_+$ the map $\mathrm{T}(A) \ni \tau \mapsto d_\tau (a) \in [0, 1]$ (called the \emph{rank} of $a$) is lower semicontinuous and affine \cite{Thi20}. Moreover, if $a$ is locally invertible, the rank of $a$ is actually continuous.

\begin{lemma}\label{lem not locally invertible element decreasing sequences}
Let $A$ be a simple, unital $\mathrm{C}^*$-algebra with $\mathrm{QT}(A)\ne\emptyset$ and let $a\in A_+$. If $a$ is not locally invertible, then there are a strictly decreasing sequence $\varepsilon_k\downarrow0$ and an increasing sequence $(a_k)_{k\ge1}\sub \mathrm{C}^*(a)_{+}^1$ such that $d_\tau(a-\varepsilon_k)_+<\tau(a_k)\le d_\tau(a-\varepsilon_{k+1})_+$ for all $k\in\bb{N}$, $\tau\in\mathrm{QT}(A)$.
\end{lemma}

\begin{proof}
Since $0\in\sigma(a)$ is not isolated, let $(\eta_k)_{k\ge1}\sub\sigma(a)$ be a strictly decreasing sequence with $\eta_k\to0$. For $k\in\bb{N}$ consider the continuous functions $f_k\colon\bb{R}\to\bb{R}$ defined as
\begin{equation*}
f_k(t) \coloneqq \begin{cases} (t-\eta_{3k+2})_+,\quad t\le\eta_{3k+1} \\ \text{linear},\quad t\in[\eta_{3k+1},\eta_{3k}]\\ 0,\quad t\ge\eta_{3k}
\end{cases}
\end{equation*}
and
\begin{equation*}
g_k(t)\coloneqq \begin{cases}0,\quad t\le\eta_{3k+3}\\ \text{linear},\quad t\in[\eta_{3k+3},\eta_{3k+2}]\\ 1,\quad t\ge\eta_{3k+2} \end{cases}
\end{equation*}
Note that $(t-\eta_{3k})_++f_k(t)\le(t-\eta_{3k+3})_+$ and $(t-\eta_{3k})_+\cdot f_k(t)=0$ for all $t\in[0,1]$, and $f_k(\eta_{3k+1})=\eta_{3k+1}-\eta_{3k+2}>0$. Clearly, $(t-\eta_{3k})_+^{1/n}+f_k(t)^{1/m}\le g_k(t)$ for all $m,n,k\in\bb{N}$, so $d_\tau(a-\eta_{3k})_++d_\tau(f_k(a))\le\tau(g_k(a))$ for all $k\in\bb{N}$ and all $\tau\in\mathrm{QT}(A)$. Since $\eta_{3k+1}\in\sigma(a)$, the element $f_k(a)\in \mathrm{C}^\ast(a)_+$ is non-zero, and since $A$ is simple, $d_\tau(f_k(a))>0$, which is to say that $d_\tau(a-\eta_{3k})_+<\tau(g_k(a))$ for all $k\in\bb{N}$ and $\tau\in\mathrm{QT}(A)$. Finally, since $g_k$ has the same support as $(t-\eta_{3k+3})_+$, we have that $\tau(g_k(a))\le d_\tau(g_k(a))=d_\tau(a-\eta_{3k+3})_+$ for all $\tau\in\mathrm{QT}(A)$. Setting $a_k \coloneqq g_k(a)$ and $\varepsilon_k \coloneqq \eta_{3k}$ for all $k\in\bb{N}$ yields the statement.
\end{proof}

The following folklore lemma is well-known and, using compactness of $\mathrm{QT}(A)$ under the assumption of $A$ being unital, its proof is a direct application of Dini's theorem (see e.g.\ \cite[Remark~2.7]{Win12}).

\begin{lemma}\label{lem dim functions ctdns}
Let $A$ be a unital $\mathrm{C}^*$-algebra with $\mathrm{QT}(A)$ nonempty. If $a,b\in A_+$ satisfy $d_\tau(a)<d_\tau(b)$ for all $\tau\in\mathrm{QT}(A)$, then for any $\varepsilon>0$ there exists $\delta>0$ so that $d_\tau(a-\varepsilon)_+<d_\tau(b-\delta)_+$ for all $\tau\in\mathrm{QT}(A)$.
\end{lemma}

We will also make use of the following interpolation property from Choquet theory which we recall here for convenience. 

\begin{proposition}{\normalfont{\cite[Theorem~11.12]{Goo10}}}\label{prop:choquetinterpol}
Let $K$ be a Choquet simplex and let $f,g\colon K \to \bb{R}$ be affine maps with $f$ upper semicontinuous and $g$ lower semicontinuous. If $f(\kappa) < g(\kappa)$ for all $\kappa \in K$, then there is a continuous affine map $h \colon K \to \bb{R}$ such that $f(\kappa) < h(\kappa) <g(\kappa)$ for all $\kappa \in K$.
\end{proposition}

\begin{lemma}\label{lem:herexpstable}
Let $(D \subset A)$ be a $\mathrm{C}^\ast$-pair with $D$ abelian, and let $\Phi\colon A \to D$ be a hereditary positive linear map. Then, the map $\Phi\otimes E_n\colon A\otimes M_n \to D\otimes D_n$ is hereditary for any $n\in\bb{N}$.
\end{lemma}

\begin{proof}
Note first that if $a,b\in A_+$ and $x\in A$ are such that $xx^*\in\her(a)$ and $x^*x\in\her(b)$, then $x\in\overline{aAb}$. We briefly prove this: for $\varepsilon>0$, we have $(xx^*)^{\varepsilon/2}\in\her(a)$, so we can write $(xx^*)^{\varepsilon/2}=\lim_n aa_n$ for some $(a_n)_{n\ge1}\sub A$ and likewise, we can write $(x^*x)^{\varepsilon/2}=\lim_nb_nb$ for some $(b_n)_{n\ge1}\sub A$. Now since $f(xx^*)x=xf(x^*x)$ for all $f\in C_0(\sigma(xx^*))$, we have
\begin{align*}
(xx^*)^\varepsilon x &= (xx^*)^{\varepsilon/2}(xx^*)^{\varepsilon/2}x \\ 
&= (xx^*)^{\varepsilon/2}x(x^*x)^{\varepsilon/2}\\
&=\lim_{n\to\infty} aa_nb_nb\in\overline{aAb},
\end{align*}
whence $x\in\overline{aAb}$, since $x=\lim_{\varepsilon\to0^+}(xx^*)^\varepsilon x$.

Let $n\in\bb{N}$ and let $\{e_{i,j}\}_{i,j=1}^n$ be the canonical system of matrix units for $M_n$. For any $d = \sum_{i=1}^nd_i\otimes e_{i,i}\in D\otimes D_n$ with $d_i\in D_+$, $i=1,\dots,n$, one directly sees that
\begin{equation}\label{eq hereditary in matrix}
\her(d)=\overline{\mathrm{span}}\big\{a\otimes e_{i,j}:i,j=1,\dots,n; a\in\overline{d_iAd_j}\big\}.
\end{equation}
Now let $x\in(A\otimes M_n)_+$, so we can write $ x = \sum_{i,j=1}^n\sum_{k=1}^na_{k,i}^*a_{k,j}\otimes e_{i,j}$ where $a_{k,j}\in A$. Set $d_i:=\sum_{k=1}^n\Phi(a_{k,i}^*a_{k,i})\in D_+$, so 
\[
d\coloneqq \Phi\otimes E_n(x)=\sum_{i=1}^nd_i\otimes e_{i,i} \in D\otimes D_n.
\]
By \eqref{eq hereditary in matrix} it suffices to show that $\sum_{k=1}^na_{k,i}^*a_{k,j}\in\overline{d_iAd_j}$ for all $i,j=1,\dots,n$. Let $1\le i,j,k\le n$. Note that since $\Phi\colon A\to D$ is hereditary, we have
\begin{align*}
0&\le(a_{k,i}^*a_{k,j})^*(a_{k,i}^*a_{k,j}) \\ 
&\le\|a_{k,i}\|^2a_{k,j}^*a_{k,j}\in\her(\Phi(a_{k,j}^*a_{k,j}))\sub\her(d_j)
\end{align*}
and thus $(a_{k,i}^*a_{k,j})^*(a_{k,i}^*a_{k,j})\in\her(d_j)$. The same argument also shows that $(a_{k,i}^*a_{k,j})(a_{k,i}^*a_{k,j})^*\in\her(d_i)$, and thus $a_{k,i}^*a_{k,j}\in\overline{d_iAd_j}$. This completes the proof, since $\sum_{k=1}^na_{k,i}^\ast a_{k,j}\in \overline{d_iAd_j}$ for all $i,j=1,\dots,n$, as we wanted.
\end{proof}

\begin{proposition}\label{prop:strict-and-dynamical-imply-diagcomp}
Let $(D \subset A)$ be a unital, ample diagonal pair with hereditary associated conditional expectation, and assume that $A$ is separable, simple and that $\mathrm{QT}(A)=\mathrm{T}(A)$ is nonempty. If $(D \subset A)$ has dynamical comparison and $A$ has strict comparison, then $(D \subset A)$ has diagonal comparison. 
\end{proposition}

\begin{proof}
If $A$ is finite-dimensional, then it is a matrix algebra due to simplicity, and diagonal comparison can be verified directly: if $a,b$ are positive matrices with the rank of $a$ being less than the rank of $b$ and $p_a, p_b$ are the (diagonal) support projections of the conditional expectations $E(a)$ and $E(b)$ respectively, then there is a permutation matrix $w$ (thus a normaliser of diagonal matrices) conjugating $p_b$ to a diagonal projection greater than $p_a$. The spectrum of $b$ is finite, so $b^{1/2}$ is locally invertible, and thus $b^{1/2}$ is also invertible in $\her(p_b)$. If $v\in\her(p_b)_+$ is the (local) inverse of $b^{1/2}$, we have that $a= a^{1/2}p_awvbv^*w^*p_aa^{1/2}$, i.e.\ $a\precsimd b$. We can thus restrict to the infinite-dimensional case, which is the nontrivial one.

Let $n\in\bb{N}$ and let $a,b\in (A\otimes M_n)_+^1$ be such that $d_\tau (a) < d_\tau(b)$ for all $\tau\in\mathrm{T}(A)$. Let $\Phi\colon A\otimes M_n \to D \otimes D_n$ denote the associated conditional expectation which is hereditary by \Cref{lem:herexpstable}. We will show that $a\precsimd b$ by distinguishing cases based on local invertibility of $a$. Let $\varepsilon>0$.

Assume first that $a$ is not locally invertible. By looking at three successive terms that are far out enough in the sequence obtained by employing \Cref{lem not locally invertible element decreasing sequences}, we find $0<\eta' < \eta <\varepsilon$ and positive elements $\tilde{a},\bar{a}\in\mathrm{C}^\ast(a)_+^1$ such that
\[
d_\tau(a-\eta)_+<\tau(\tilde{a}) < \tau(\bar{a}) < d_\tau(a-\eta')_+
\]
for all $\tau \in \mathrm{T}(A)$. Now by compactness of $\mathrm{T}(A)$ and  \Cref{cor:evprojdense} we can approximate the continuous map $\tau\mapsto \frac{1}{2}(\tau(\bar{a})+\tau(\tilde{a}))$ sufficiently well to obtain a projection $p\in D\otimes D_n$ such that 
\begin{equation}\label{eq:proofctdns}
d_\tau(a-\eta)_+ < \tau(p) <d_\tau(a-\eta')_+, \quad \tau\in\mathrm{T}(A).
\end{equation}
Since $\tau(p) < d_\tau(b)$ for all $\tau\in\mathrm{T}(A)$, by \Cref{prop:choquetinterpol} and \Cref{cor:evprojdense} we find a projection $q \in D\otimes D_n$ such that
\begin{equation}\label{eq:proofprojns}
\tau(p) < \tau(q) < d_\tau(b), \quad \tau\in \mathrm{T}(A).
\end{equation}
Since $\Phi$ is hereditary, by \eqref{eq:proofctdns} we have that $\tau(p) < d_\tau(\Phi(a))$ for all $\tau\in\mathrm{T}(A)$. By \Cref{lem dim functions ctdns} and since $p$ is a projection, there is some $\theta>0$ such that $\tau(p)<d_\tau(\Phi(a)-\theta)_+$ for all $\tau \in \mathrm{T}(A)$. By zero-dimensionality of the spectrum of $D \otimes D_n$, there is a clopen subset of $\supp(\Phi(a))$ that contains $\supp(\Phi(a)-\theta)_+$, so, if $p'\in D\otimes D_n$ denotes the indicator function of this clopen set, then $\supp(p') \subset \supp(\Phi(a))$ and $\tau(p) < \tau(p')$ for all $\tau\in\mathrm{T}(A)$. Using dynamical comparison of $(D \otimes D_n \subset A \otimes M_n)$ which is clear by \Cref{rmk:stabledynamicalcomparison}, upon replacing $p$ by a subprojection of $p'$, we can assume without loss of generality that $\supp(p) \subset \supp(\Phi(a))$. Similarly, we can assume that $\supp(q) \subset \supp(\Phi(b))$. By \eqref{eq:proofprojns} and dynamical comparison, there is $w\in\mc{N}_{A\otimes M_n}(D \otimes D_n)$ such that $p=wqw^* = pwqw^*p$. By \eqref{eq:proofctdns} and strict comparison there is $s\in A\otimes M_n$ such that $(a-\eta)_+ \approx_{(\varepsilon-\eta)/2} sps^*$. By \eqref{eq:proofprojns} and strict comparison there is $t \in A \otimes M_n$ such that $q \approx_{\delta} qtbt^*q$, where $\delta\coloneqq \frac{\varepsilon-\eta}{2\|spw\|^2+1}$. Setting $u \coloneqq sp$ and $v\coloneqq qt$, we see that $\supp(\Phi(u^*u)) \subset \supp(p) \subset \supp(\Phi(a))$ and $\supp(\Phi(vv^*)) \subset \supp(q) \subset \supp(\Phi(b))$, and that $a \approx_\varepsilon uwvbv^*w^*u^*$, which shows that $a \precsimd b$.

Assume now that $a$ is locally invertible, in which case the rank of $a$ is weak$^\ast$-continuous. Since $a \in \her(\Phi(a))$, we have that $d_\tau(a) \le d_\tau(\Phi(a))$ for all $\tau\in\mathrm{T}(A)$.

If the rank of $a$ agrees with the rank of $\Phi(a)$, then by \Cref{prop:choquetinterpol} and \Cref{cor:evprojdense} we obtain a projection $q \in D\otimes D_n$ with $d_\tau(\Phi(a)) < \tau(q) < d_\tau(b)$, $\tau\in\mathrm{T}(A)$. Arguing as before, by dynamical comparison we can assume without loss of generality that $\supp(q) \subset \supp(\Phi(b))$. Since $a\in\her(\Phi(a))$, by zero-dimensionality of the spectrum of $D\otimes D_n$ there is a projection $p \in D\otimes D_n$ such that $\supp(p)\subset \supp(\Phi(a))$ and $\|a - pap\|<\varepsilon/3$, whence $(a-\varepsilon/3)_+ \precsim p$. Since $\tau(p) \le d_\tau(\Phi(a)) < \tau(q) < d_\tau(b)$ for all $\mathrm{T}(A)$, we can proceed exactly as in the last part of the previous case with $\eta=\varepsilon/3$ therein.

If on the other hand there is some $\tau_0\in\mathrm{T}(A)$ such that $d_{\tau_0}(a) < d_{\tau_0}(\Phi(a))$, then $\her(a)$ is a unital $\mathrm{C}^\ast$-algebra that is a proper subset of $\her(\Phi(a))$, so there is some $c \in \her(\Phi(a))$ such that $0\ne z\coloneqq (1_A - 1_{\her(a)})c \in \her(\Phi(a))$. Since $A$ is simple, $d_\tau(zz^*)>0$ for all $\tau\in\mathrm{T}(A)$, and since $a \perp zz^*$ we have $d_\tau(a) + d_\tau(zz^*) = d_\tau(a+zz^*)\le d_\tau(\Phi(a))$, which is to say that $d_\tau(a) < d_\tau(\Phi(a))$ for all $\tau\in\mathrm{T}(A)$. Now by \cite[Lemma~3.7]{Thi20} there is a lower semicontinuous, affine map $f\colon \mathrm{T}(A)\to [0,1]$ satisfying
\begin{equation}\label{eq:infmap}
f(\tau)\le \min\{d_\tau(\Phi(a)) , d_\tau(b) \}, \quad \tau\in \mathrm{T}(A),
\end{equation}
and equality is achieved in \eqref{eq:infmap} on extremal traces. By \cite[Proposition~3.4]{Thi20}, since the rank of $a$ is less than $f$ on the extreme boundary of $\mathrm{T}(A)$, we actually have $d_\tau(a) < f(\tau)$ for all $\tau \in \mathrm{T}(A)$. Since the rank of $a$ is continuous, by \Cref{prop:choquetinterpol} and \Cref{cor:evprojdense}, we find a projection $p \in D\otimes D_n$ such that $d_\tau(a) < \tau (p) < f(\tau)$, for which, by dynamical comparison, we can assume without loss of generality that $\supp(p) \subset \supp(\Phi(a))$. Since $\tau(p) < d_\tau(b)$, again by \Cref{prop:choquetinterpol} and \Cref{cor:evprojdense} we find a projection $q\in D\otimes D_n$ with $\tau(p) < \tau(q) <d_\tau(b)$, for which we may assume that $\supp(q) \subset \supp(\Phi(b))$, and the proof now concludes as in the previous case.

\end{proof}
Combining Proposition~\ref{prop:strict-and-dynamical-imply-diagcomp} together with Proposition~\ref{prop:diagcomptodyncomp}, we obtain our main result; cf. \Cref{intro main thm}.

\begin{theorem}\label{thm:main}
Let $(D \subset A)$ be a unital, ample diagonal pair with hereditary associated conditional expectation, and assume that $A$ is separable, simple and that $\mathrm{QT}(A)=\mathrm{T}(A)$ is nonempty. The following are equivalent:
\begin{enumerate}[label=\normalfont(\roman*)]
\item $(D \subset A)$ has dynamical comparison and $A$ has strict comparison.
\item $(D \subset A)$ has diagonal comparison.
\end{enumerate}
\end{theorem}

As a corollary, we obtain relations between diagonal comparison, diagonal dimension (cf.\ \cite[Definition~2.1]{LiLiaWin23}), and tracial $\mathcal{Z}$-stability of a pair (cf.\ \cite[Definition~3.2]{LiaTik22}), in the case that the pair arises from a free minimal action of an amenable group on a Cantor space.

\begin{corollary}
Let $G \acts X$ be a free minimal action of a countable amenable group on a compact, infinite, zero-dimensional metrisable space. Consider the following statements:
\begin{enumerate}[label=\normalfont(\roman*)]
\item\label{itc1} $\dim_{\mathrm{diag}}(C(X)\subset C(X)\rtimes_\mathrm{r} G) <\infty$.
\item\label{itc2} $(C(X) \subset C(X)\rtimes_\mathrm{r} G)$ is tracially $\mathcal{Z}$-stable.
\item\label{itc2.5} $G\acts X$ has dynamical comparison. 
\item\label{itc3} $(C(X) \subset C(X)\rtimes_\mathrm{r} G)$ has diagonal comparison.
\end{enumerate}
Then \ref{itc1}$\implies$\ref{itc2}$\iff$\ref{itc2.5}$\iff$\ref{itc3}.
\end{corollary}

\begin{proof}
\ref{itc2}$\iff$\ref{itc2.5} is \cite[Corollary~3.7]{LiaTik22}. The equivalence \ref{itc2.5}$\iff$\ref{itc3} follows from \Cref{thm:main} which is applicable since the conditional expectation $C(X)\rtimes_\mathrm{r} G \to C(X)$ is hereditary (see \Cref{rmk:herexponcrossedprod}), together with the fact that dynamical comparison of $G \acts X$ and zero-dimensionality of $X$ imply strict comparison of $C(X)\rtimes_\mathrm{r} G$ by \cite[Theorem~12.4]{Ker20}. Finally, if $(C(X) \subset C(X)\rtimes_\mathrm{r} G)$ has finite diagonal dimension, then by \cite[Theorem~5.4]{LiLiaWin23} and \cite[Theorem~7.2]{Ker20} $G \acts X$ has dynamical comparison, so \ref{itc1}$\implies$\ref{itc2.5}.

\end{proof}

\section{Hereditary conditional expectations}\label{sec:herexp}
\noindent In this final section we explore further the question of when a conditional expectation is hereditary. The following auxiliary proposition is proved essentially as \cite[Proposition~4.2]{Win12}, and so we omit the details (see also \cite[Remark~2.2~(iii)]{LiLiaWin23}).

\begin{proposition}\label{lemma: approx commuting coordinates}
Let $(D\subset A)$ be a unital $\mathrm{C}^*$-pair with $\dim_\mathrm{diag}(D\subset A)=d<\infty$. Then, given $a\in A_{+}^1$ and $\delta>0$, there exist a finite-dimensional $\mathrm{C}^*$-algebra $F=\bigoplus_{j=0}^dF^{(j)}$ with a masa $D_F=\bigoplus_{j=0}^dD_{F^{(j)}} \subset F$ and maps
\[
\begin{tikzcd}
A\ar[dr,"\psi"]&&A\\
&F\ar[ur,"\varphi"]&
\end{tikzcd}
\]
such that
\begin{enumerate}[label=\normalfont(\roman*)]
\item\label{itddlem1} $\psi$ is completely positive and contractive (c.p.c.),
\item\label{itddlem2} the composition $\varphi\psi$ is contractive,
\item\label{itddlem3} $\|\varphi\psi(a)-a\|<\delta$,
\item\label{itddlem4} $\psi(D)\subset D_F$,
\item\label{itddlem5} $\varphi^{(j)}:=\varphi\vert_{F^{(j)}}$ is c.p.c.\ order zero and $\varphi^{(j)}(\mc{N}_{F^{(j)}}(D_{F^{(j)}}))\subset\mc{N}_A(D)$ for all $j=0,\dots,d$,
\item\label{itddlem6} for any $j=0,\dots,d$, and any matrix summand $E$ appearing in $F^{(j)}$, we have
\[
\|\varphi_E^{(j)}\psi_E^{(j)}(a)-a^{1/2}\varphi_E^{(j)}\psi_E^{(j)}(1_A)a^{1/2}\|<\delta,
\]
\end{enumerate}
with $\varphi_E^{(j)} \coloneqq \varphi^{(j)}\vert_E$ and $\psi^{(j)}_E \coloneqq \pi_E\circ\psi$, where $\pi_E$ is the canonical surjection $F\to E$.
\end{proposition}

\begin{remark}\label{rmk:cpc-order-zero-functional-calculus}
For the proof of \Cref{prop:ddher} below we will need that order zero maps and order zero functional calculus are compatible with preservation of normalisers: if $(D_A \subset A)$ and $(D_B \subset B)$ are unital sub-$\mathrm{C}^\ast$-algebras with $D_A, D_B$ abelian and $\varphi\colon A\to B$ is a c.p.c.\ order zero map that preserves normalisers, then we automatically have $\varphi(D_A) \subset D_B$ (cf.\ \cite[Proposition~1.9~(i)]{LiLiaWin23} or \cite[Lemma~1.6]{LiaTik22}). Moreover, for a positive $f\in C_0(0,1]$, the c.p.c.\ order zero map $f(\varphi)\colon A\to B$ defined by order zero functional calculus (see \cite[Corollary~3.2]{WinZac09}) preserves normalisers: indeed, let $\pi_\varphi$ denote the $^*$-homomorphism associated to $\varphi$ by the structure theorem of order zero maps \cite[Theorem~2.3]{WinZac09}, so $\varphi(.)=\varphi(1_A)\pi_\varphi(.)$. For $v\in\mc{N}_A(D_A)$, we have $f(\varphi)(v)=f(\varphi(1_A))\cdot \pi_\varphi(v)$. Uniformly approximating $f$ by polynomials with zero constant term, we can find a sequence $(d_n)_{n=1}^\infty \subset D_B$ such that $f(\varphi(1_A))=\lim_{n\to\infty} d_n\cdot \varphi(1_A)$, whence $f(\varphi)(v) = \lim_{n\to\infty} d_n \varphi(1_A)\pi_\varphi(v)=\lim_{n\to\infty} d_n \varphi(v)\in\mc{N}_B(D_B)$.
\end{remark}

\begin{proposition}\label{prop:ddher}
Let $(D \subset A)$ be a unital diagonal pair with associated conditional expectation $\Phi\colon A \to D$. If $\dim_\mathrm{diag}(D \subset A) < \infty$, then $\Phi$ is hereditary. 
\end{proposition}

\begin{proof}
Let $a\in A_+$ with $\|a\|=1$ and let $\varepsilon>0$. We will find $f\in \her(\Phi(a))_{+}^1$ such that $\|a-faf\|<\varepsilon$. Take $\eta>0$ small enough so that $2\eta+2\sqrt{2\eta}<\varepsilon$. By \Cref{lemma: approx commuting coordinates} we obtain a finite-dimensional $\mathrm{C}^\ast$-algebra $F=\bigoplus_{j=0}^dF^{(j)}$ with a masa $D_F=\bigoplus_{j=0}^dD_{F^{(j)}}$ and c.p.\ maps $A\xrightarrow{\psi}F\xrightarrow{\varphi}A$ satisfying conditions \ref{itddlem1}--\ref{itddlem6} there, with $\delta \coloneqq \frac{\eta}{2(d+1)}$. Moreover, we identify $F$ with a direct sum of matrix algebras so that the restriction of $D_F$ on any matrix summand of $F$ is just the diagonal matrices therein (cf.\ \cite[Example~1.5]{LiLiaWin23}).

Let $r_{\max} \in\bb{N}$ be the maximum matrix size among the matrix summands of $F$. By Urysohn's lemma there is $f\in D_{+}^1$ such that $\supp(f)\subset\supp(\Phi(a))$ and $f=1$ on the subset of the spectrum of $D$ where $\Phi(a)\ge\frac{\eta}{2(d+1)r_{\max}}$, and in particular $f\in\her(\Phi(a))$. Set $f'\coloneqq 1-f\in D_{+}^1$. Note that $a=f a f + f' a f + f a f' + f' a f'$ and therefore
\begin{align*}
\|a-faf\|& \le\|f' a f'\|+2 \|faf'\| \\
&\le \|f' a f'\| +2 \|a^{1/2}f'\| \\
&= \|f' a f'\| +2\|f' a f'\|^{1/2},
\end{align*}
whence, by our choice of $\eta$, it suffices to show that $\|f' \varphi\psi(a) f'\|<\eta$, since then we also have $\|f'af'\|<2\eta$, and thus $\|a-faf\|<2\eta + 2\sqrt{2\eta}<\varepsilon$. Setting $\psi^{(j)}\coloneqq\pi_{F^{(j)}}\circ\psi$ where $\pi_{F^{(j)}}$ is the canonical surjection $F\to F^{(j)}$, since $\varphi\psi=\sum_{j=0}^d\varphi^{(j)}\psi^{(j)}$, it further suffices to show that $\|f' \varphi^{(j)}\psi^{(j)}(a) f'\|<\frac{\eta}{d+1}$ for all $j=0,\dots,d$, since then it follows by the triangle inequality that 
\begin{equation*}
\|f' \varphi\psi(a) f'\|\le\sum_{j=0}^d\|f' \varphi^{(j)}\psi^{(j)}(a) f'\|<\sum_{j=0}^d\frac{\eta}{(d+1)}=\eta.
\end{equation*}
Fix $j_0\in\{0,\dots,d\}$ and set $B\coloneqq \mathrm{C}^*(\varphi^{(j_0)}(F^{(j_0)}), D)\subset A$. For a matrix summand $E$ appearing in $F^{(j_0)}$ let $J_E:=\langle\varphi_E^{(j_0)}(E)\rangle\trianglelefteq B$ be the ideal generated by $\varphi_E^{(j_0)}(E)$ in $B$. (These ideals as well as $B$ itself are very well-behaved since $\varphi^{(j_0)}$ is compatible with $D$; in fact, $B$ and the ideals $J_E$ are subhomogeneous, as we will see below.)

If $E, K$ are distinct matrix summands appearing in $F^{(j_0)}$, then some calculations show that $J_EJ_K=\{0\}$: to see this, by definition of $B$ and since $1_A\in D$, it suffices to show that for any $n\ge0$, any matrix summands $E_1,\dots, E_n$ of $F^{(j_0)}$, any matrix units $e_1\in E_1, \dots, e_n\in E_n$, any matrix units $e\in E, e'\in K$, and any $f_1,\dots, f_{n+1}\in D$, we have that
\begin{equation}\label{eq:rev1}
z\coloneqq \varphi^{(j_0)}(e)\cdot \bigg( f_1 \varphi^{(j_0)}(e_1) f_2 \varphi^{(j_0)}(e_2) \cdots f_n \varphi^{(j_0)}(e_n) f_{n+1}\bigg) \cdot \varphi^{(j_0)}(e') = 0,
\end{equation}
where, if $n=0$, the middle product is to be interpreted as only $f_{n+1}$. Note that since matrix units of any matrix summand in $F^{(j_0)}$ are normalisers of $D_{F^{(j_0)}}$, by \ref{itddlem5} of our application of \Cref{lemma: approx commuting coordinates} we have that each of the $\varphi^{(j_0)}(e_i)$ and $\varphi^{(j_0)}(e),\varphi^{(j_0)}(e')$ are in $\mc{N}_A(D)$. Repeatedly using this observation together with the fact that $D$ is commutative, one sees that for $zz^*$ to be non-zero it is necessary that $K=E_n=\dots=E_1=E$, which is impossible.

Since $f' \varphi^{(j_0)}\psi^{(j_0)}(a) f' =\sum_E f' \varphi_E^{(j_0)}\psi^{(j_0)}_E(a) f'$ which is a sum of pairwise orthogonal contractions, we have that
\begin{equation}\label{eq:max matrix summands}
\|f' \varphi^{(j_0)}\psi^{(j_0)}(a) f'\|=\max_E \|f'\varphi^{(j_0)}_E\psi^{(j_0)}_E(a) f'\|.
\end{equation}
Let $E=M_r$ be a matrix summand appearing in $F^{(j_0)}$ and set $h_E\coloneqq f' \varphi_E^{(j_0)}\psi_E^{(j_0)}(a) f'\in J_E$. To simplify notation, set $\varphi_E \coloneqq \varphi_E^{(j_0)}$. As $E$ is simple, the structure theorem of order zero maps yields that $\varphi_E$ is either zero or injective. In the first case $h_E=0$ and there is nothing to do, so assume that $\varphi_E$ is injective.

If $\{e_{i,j}\}_{i,j=1}^r$ are the matrix units of $E$, we see that $\overline{\varphi_E(e_{i,i}) B \varphi_E(e_{i,i})}\subset D$ for all $i=1,\dots,r$. Indeed, similarly to \eqref{eq:rev1}, it suffices to show that for any $n\ge0$, any matrix units $e_{i_1,j_1},\dots,e_{i_n,j_n}$ of matrix summands of $F^{(j_0)}$, and any $f_1,\dots,f_{n+1}\in D$, we have that
\begin{equation}\label{eq:rev2}
\varphi_E(e_{i,i})\cdot \bigg( f_1 \varphi^{(j_0)}(e_{i_1,j_1}) f_2 \varphi^{(j_0)}(e_{i_2,j_2})\cdots f_n \varphi^{(j_0)}(e_{i_n,j_n}) f_{n+1}\bigg) \cdot \varphi_E(e_{i,i})\in D,
\end{equation}
where if $n=0$, the middle product is to be interpreted as only $f_{n+1}$ (and in which case \eqref{eq:rev2} is obvious). Let $w$ denote the product in \eqref{eq:rev2}. If $w$ is non-zero, then all of $\varphi_E(e_{i,i})f_1\varphi^{(j_0)}(e_{i_1,j_1})$, $\varphi^{(j_0)}(e_{i_1,j_1})f_2\varphi^{(j_0)}(e_{i_2,j_2})$, $\dots$, $\varphi^{(j_0)}(e_{i_{n-1},j_{n-1}}) f_n \varphi^{(j_0)}(e_{i_n,j_n})$, $\varphi^{(j_0)}(e_{i_n,j_n})f_{n+1}\varphi_E(e_{i,i})$ are non-zero. By repeated applications of the $\mathrm{C}^\ast$-identity, the fact that all of the $\varphi^{(j_0)}(e_{i_k,j_k})$ are normalisers of $D$ and since $\varphi^{(j_0)}$ is order zero, it follows that each $e_{i_k,j_k}$ is a matrix unit of $E$, and moreover $i_1=i$, $i_{k+1}=j_k$ for all $k=1,\dots,n-1$, and $j_n=i$. Now using order zero functional calculus, set $\beta\coloneqq(\varphi^{(j_0)})^\frac{1}{2}$, and observe that $\varphi^{(j_0)}(xy)=\beta(x)\beta(y)$ for all $x,y\in F^{(j_0)}$. Putting these together, and since $\beta$ preserves normalisers by \Cref{rmk:cpc-order-zero-functional-calculus}, we see that
\begin{align*}
\varphi^{(j_0)}(e_{i_1,j_1}) f_2 \varphi^{(j_0)}(e_{i_2,j_2})\cdots \varphi^{(j_0)}(e_{i_{n-1},j_{n-1}}) f_n \varphi^{(j_0)}(e_{i_n,j_n}) &= \\
\varphi^{(j_0)}(e_{i,j_1}) f_2 \varphi^{(j_0)}(e_{j_1,j_2})\cdots \varphi^{(j_0)}(e_{j_{n-2},j_{n-1}}) f_n \varphi^{(j_0)}(e_{j_{n-1},i}) &= \\
\underbrace{\beta(e_{i,i})}_{\in D}\cdot \underbrace{\beta(e_{i,j_1}) f_2 \beta(e_{j_1,i})}_{\in D}\cdot \underbrace{\beta(e_{i,j_2}) f_3 \beta(e_{j_2,i})}_{\in D} \cdots \\
\underbrace{\cdots  \beta(e_{j_{n-2},i})}_{\in D}\cdot \underbrace{\beta(e_{i,j_{n-1}}) f_n \beta(e_{j_{n-1},i})}_{\in D}\cdot \underbrace{\beta(e_{i,i})}_{\in D},&
\end{align*}
and therefore \eqref{eq:rev2} shows that $w\in D$.

Let $\pi\colon J_E\to\bh$ be an irreducible representation of $J_E$ such that $\|h_E\|=\|\pi(h_E)\|$. Note that $\varphi_E(1_E)$ commutes with $B$ (due to the structure theorem of order zero maps and the fact that $\varphi_E(1_E)\in D$), and in particular $\pi(\varphi_E(1_E))\in \pi(J_E)'=\mathbb{C}\mathrm{id}_\mc{H}$, so there is $\lambda >0$ such that $\pi(\varphi_E(1_E))=\lambda \cdot \mathrm{id}_\mc{H}$. Now $\theta\coloneqq \frac{1}{\lambda}\cdot  \pi\circ\varphi_E$ is a unital order zero c.p.\ map, whence it is a unital $^*$-homomorphism $M_r\to\bh$, due to the structure theorem of order zero maps. The projections $p_i\coloneqq \theta(e_{i,i})$ are pairwise orthogonal and satisfy $\sum_{i=1}^rp_i=\mathrm{id}_\mc{H}$, and for $x\in \overline{\varphi_E(e_{i,i})J_E\varphi_E(e_{i,i})}$ one has $\pi(x)=p_i\pi(x)=\pi(x)p_i$. Therefore, restricting $\pi$ induces representations $\overline{\varphi_E(e_{i,i})J_E\varphi_E(e_{i,i})}\to p_i\bh p_i\cong \mathbb{B}(p_i\mc{H})$, and we will show that these are irreducible. Before proving this, observe that a consequence of this is that, since each $\overline{\varphi_E(e_{i,i})J_E\varphi_E(e_{i,i})}$ is a non-zero and abelian $\mathrm{C}^\ast$-algebra (as a subset of $D$), each projection $p_i$ must have rank one, and so $\mc{H}$ is $r$-dimensional. It follows that $\bh$ can be identified with $M_r$ in such a way so that $\theta$ becomes the identity map, i.e.\
\begin{equation}\label{eq multiple of id}
\pi\circ\varphi_E=\lambda\cdot\mathrm{id}_E.
\end{equation}

To verify that the restriction of $\pi$ on $\overline{\varphi_E(e_{i,i})J_E\varphi_E(e_{i,i})}$ is irreducible, let $t\in p_i\bh p_i$ be in the commutant of $\pi(\overline{\varphi_E(e_{i,i})J_E\varphi_E(e_{i,i})})$, and set 
\[
\tilde{t}\coloneqq \sum_{j=1}^r\theta(e_{j,i})t\theta(e_{i,j})\in \bh .
\]
For any $y\in J_E$, we have $\pi(y) = \sum_{j,k=1}^r p_j \pi(y) p_k$ and $\theta(e_{i,j})\pi(y)\theta(e_{k,i})\in\pi(\overline{\varphi_E(e_{i,i})J_E\varphi_E(e_{i,i})})$. A direct calculation using these facts shows that $\tilde{t}$ commutes with $\pi(J_E)$. Since $\pi$ is irreducible, $\tilde{t}\in \mathbb{C} \mathrm{id}_\mc{H}$, and thus $t= p_i\tilde{t}p_i \in \mathbb{C} p_i$, as we wanted.

%Using the aforementioned identification of $\bh$ and $M_r$, the composition $\pi\circ\varphi_E\colon M_r\to M_r$ is a bijective c.p.c.\ order zero map and maps $e_{i,i}$ to a scalar multiple of itself. It follows by the structure theorem that there is $\lambda>0$ such that 

Let $\sigma\colon \overline{\varphi_E(e_{i,i})J_E\varphi_E(e_{i,i})}\to\bb{C}$ be the character satisfying 
\[
\pi(x)=\sigma(x)\cdot e_{1,1}
\] 
for all $x\in \overline{\varphi_E(e_{i,i})J_E\varphi_E(e_{i,i})}$. Since $\overline{\varphi_E(e_{i,i})J_E\varphi_E(e_{i,i})}\subset D$, we can extend $\sigma$ to a character on $D$ which we denote again by $\sigma$, and which we can in turn extend to a state on $A$ as $\tilde{\sigma}:=\sigma\circ\Phi$. Define a map $\varrho\colon A\to M_r$ as
\begin{equation}
\varrho(.)\coloneqq\frac{1}{\lambda^2}\sum_{i,j=1}^r\tilde{\sigma}(\varphi_E(e_{1,i})\; . \;\varphi_E(e_{j,1})) e_{i,j}.
\end{equation}

It is straightforward to check that $\varrho$ is a unital, c.p.\ map. Since $\varphi_E$ is normaliser-preserving by \ref{itddlem5} of \Cref{lemma: approx commuting coordinates} and since $e_{1,i}\in \mc{N}_{M_r}(D_r)$, and using that $n\Phi(.)n^*=\Phi(n\; .\; n^*)$ for any $n\in\mc{N}_A(D)$ \cite[Lemma~6$^\text{o}$]{Kum86} (see also \cite[Lemma~3.2]{CryNag17}), for $b\in A_+$ we have that the diagonal entries of the matrix $\varrho(b)=[\varrho(b)_{i,j}]_{i,j=1}^r$ satisfy
\begin{align}\label{eq:diagentriesconditionalexp}
\begin{split}
0\le\varrho(b)_{i,i} & = \lambda^{-2}\cdot \sigma(\Phi(\varphi_E(e_{1,i})b\varphi_E(e_{i,1}))) \\
 &= \lambda^{-2}\cdot \sigma(\varphi_E(e_{1,i}) \Phi(b) \varphi_E(e_{i,1}) ) \\
 & \le \lambda^{-2}\cdot \|\Phi(b)\|\cdot\sigma(\varphi_E(e_{1,i})\varphi_E(e_{i,1})) \\
 & = \|\Phi(b)\|.
\end{split}
\end{align}

We claim that, for $g\in D_+$ and $e_{i,j}\in M_r$ we have $\varrho(g\varphi_E(e_{i,j})g)=\pi(g\varphi_E(e_{i,j})g)$. Indeed, using that $\varphi_E(D_r)$ is a subset of $D$ which is abelian, by order zero functional calculus we see that $g\varphi_E(e_{i,j})g \in \overline{ \varphi_E(e_{i,i}) A \varphi_E(e_{j,j})}$ and therefore
\begin{equation}\label{eq:rho-extends-pi}
\varrho(g\varphi_E(e_{i,j})g)=\frac{1}{\lambda^2}\tilde{\sigma}  (\varphi_E(e_{1,i})g\varphi_E(e_{i,j})g\varphi_E(e_{j,1}))\cdot e_{i,j}.
\end{equation}
Now since $\varphi_E(e_{1,i})g\varphi_E(e_{i,j})g\varphi_E(e_{j,1})\in J_E \cap\her(\varphi_E(e_{1,1}))\subset D$, we have that 
\begin{align*}
\sigma(\varphi_E(e_{1,i})g\varphi_E(e_{i,j})g\varphi_E(e_{j,1}))\cdot e_{1,1} &= \pi(\varphi_E(e_{1,i})g\varphi_E(e_{i,j})g\varphi_E(e_{j,1})) \\
& =\pi(\varphi_E(e_{1,i})) \pi(g\varphi_E(e_{i,j})g) \pi(\varphi_E(e_{j,1}))  \\
&=\lambda^2 \cdot e_{1,i}\cdot \pi(g\varphi_E(e_{i,j})g)\cdot e_{j,1},
\end{align*}
and thus
\begin{align*}
\sigma(\varphi_E(e_{1,i})g\varphi_E(e_{i,j})g\varphi_E(e_{j,1}))\cdot e_{i,j} &= \lambda^2\cdot  \pi(g\varphi_E(e_{i,j})g),
\end{align*}
which in conjunction with \eqref{eq:rho-extends-pi} shows that $\varrho(g\varphi_E(e_{i,j})g)=\pi(g\varphi_E(e_{i,j})g)$ as we wanted. It follows now that $\varrho(h_E)=\pi(h_E)$. Since $\pi$ was picked so that $\|h_E\|=\|\pi(h_E)\|$, we have that $\|h_E\|=\|\varrho(h_E)\|$. Let $h'_E:=f' a^{1/2}\varphi_E^{(j_0)}\psi_E^{(j_0)}(1_A) a^{1/2} f'$. Since $\|h_E-h'_E\|<\frac{\eta}{2(d+1)}$ by \ref{itddlem6} of \Cref{lemma: approx commuting coordinates} and since $\varrho$ is contractive, we have that
\begin{align}\label{eq:finitedd her}
\begin{split}
\|h_E\| &= \|\varrho(h_E)\| \\
&= \|\varrho(h_E-h'_E)+\varrho(h'_E)\| \\
&\le \frac{\eta}{2(d+1)} +\|\varrho(h'_E)\|.
\end{split}
\end{align}
Using in the first line of the chain of inequalities below that $\varrho$ is positive and $0\le h'_E\le f'af'$ (as $\varphi\psi$ is contractive), we also have that
\begin{align*}
\|\varrho(h'_E)\| &\le \|\varrho(f' a f')\|_{M_r}\\
&\le r\cdot\max_{1\le i \le r}|\varrho(f' a f')_{i,i}| \\
&\stackrel{\mathmakebox[\widthof{=}]{{\eqref{eq:diagentriesconditionalexp}}}}{\le} r\cdot \|\Phi(f' af')\| \\
&= r\cdot \|f'\Phi(a)f'\| \\ 
&< r\cdot\frac{\eta}{2(d+1)r_{\max}} \\
&\le \frac{\eta}{2(d+1)}
\end{align*}
and combining this estimate with \eqref{eq:finitedd her} yields that $\|h_E\|<\eta/(d+1)$. Since $E$ was an arbitrary matrix summand of $F^{(j_0)}$, \eqref{eq:max matrix summands} implies that $\|f' \varphi^{(j_0)}\psi^{(j_0)}(a) f'\|<\eta/(d+1)$, and since this holds for any $j_0\in\{0,\dots,d\}$, the proof is complete.
\end{proof}

\begin{remark}
It is plausible that the conditional expectation of a diagonal pair $(D \subset A)$ is automatically hereditary when $A$ is assumed to be nuclear. On the other hand, hereditariness cannot be expected to hold automatically without any amenability assumptions: by \cite[Corollary~5.1]{RoeWil14}, any non-exact discrete group $G$ admits a free action on a space $X$ such that the induced action $G\acts C(X)$ fails to be exact in the sense of \cite[Definition~1.2]{Sie10}. By \cite[Proposition~1.3]{Sie10}, there exists some positive element in the crossed product $C(X)\rtimes_\mathrm{r}G$ that does not belong to the ideal generated by its conditional expectation, and in particular the conditional expectation fails to be hereditary. Regularity of $(D \subset A)$ cannot be relaxed either, as the following example illustrates.

\end{remark}

\begin{example}\label{ex:nonhereditary}
We give an example of a unital $\mathrm{C}^*$-pair $(D\sub A)$ where $A$ is separable and AF, $D$ has zero-dimensional spectrum and is a masa in $A$, $(D\sub A)$ has the unique extension property, the induced conditional expectation $\Phi\colon A\to D$ is faithful and there exists a projection $p\in A$ so that $p\not\in\her(\Phi(p))$. The pair $(D \subset A)$ described below is not regular.

Consider the $\mathrm{C}^*$-diagonal $(D_{2^\infty}\sub M_{2^\infty})$ where $M_{2^\infty}$ is the UHF $\mathrm{C}^\ast$-algebra of type $2^\infty$ and $D_{2^\infty}$ is its standard diagonal. Consider also the canonical (unitally contained) sub-$\mathrm{C}^\ast$-algebras
\begin{equation*}
\begin{tikzcd}
D_2\ar[d, hook]\ar[r, hook] & \dots\ar[r]& D_2^{\otimes k} \ar[d, hook]\ar[r, hook] & D_2^{\otimes (k+1)} \ar[d, hook]\ar[r, hook] & \dots \ar[r, hook] & D_{2^\infty} \ar[d, hook] \\
M_2\ar[r, hook] & \dots\ar[r]& M_2^{\otimes k}\ar[r, hook] & M_2^{\otimes (k+1)}\ar[r, hook] & \dots \ar[r, hook] & M_{2^\infty}.
\end{tikzcd}
\end{equation*}
Let $E\colon M_{2^\infty}\to D_{2^\infty}$ denote the associated faithful conditional expectation and observe that $E\vert_{M_2^{\otimes k}}$ coincides with $E_2^{\otimes k}\colon M_2^{\otimes k}\to D_2^{\otimes k}$.

Consider now the $\mathrm{C}^*$-algebra $B:=\ell^\infty(\bb{N},M_{2^\infty})$ of bounded sequences over $M_{2^\infty}$ and let $D\subset B$ denote the (abelian) sub-$\mathrm{C}^\ast$-algebra of convergent sequences over $D_{2^\infty}$, namely $D:=C(\bb{N}^{(+)},D_{2^\infty})$, where $\bb{N}^{(+)}=\bb{N}\sqcup\{\infty\}$ is the one point compactification of $\bb{N}$. We define the projection $p=(p_n)_{n\ge1}\in B$, where $p_n\in M_2^{\otimes n}\subset M_{2^\infty}$ is the projection given by 
\begin{equation*}
p_n:=\frac{1}{2^n}\sum_{\bar{\imath},\bar{\jmath}\in \{1,2\}^n} e_{i_1,j_1}\otimes\dots\otimes e_{i_n,j_n}
\end{equation*}
where $\bar{\imath} = (i_1,\dots,i_n)$ and $\bar{\jmath} = (j_1,\dots,j_n)$. Let $\mathrm{C}^\ast(D,p)=:A\subset B$. Note that $(D\sub A)$ is a unital $\mathrm{C}^*$-pair where $A$ is separable and $D$ has zero-dimensional spectrum. 

We first show that $(D\sub A)$ has the unique extension property. Viewing $D$ as the tensor product of the $\mathrm{C}^*$-algebra $C(\bb{N}^{(+)})$ of convergent sequences (over $\mathbb{C}$) with $D_{2^\infty}$, we have that the set $\mathrm{PS}(D)$ of pure states on $D$ is described as
\begin{equation*}
\{\mathrm{ev}_n\otimes\psi: n\in\bb{N},\psi\in\mathrm{PS}(D_{2^\infty})\}\sqcup\{\mathrm{ev}_\infty\otimes\psi:\psi\in\mathrm{PS}(D_{2^\infty})\},
\end{equation*}
where $\mathrm{ev}_n(\lambda_j)_{j\ge1} \coloneqq \lambda_n$ for $n\in\bb{N}$ and $\mathrm{ev}_\infty(\lambda_j)_{j\ge1} \coloneqq \lim_{j\to\infty}\lambda_j$ for all $(\lambda_j)_{j\ge1}\in C(\bb{N}^{(+)})$.

Let $n\in\bb{N}$ and $\psi\in\mathrm{PS}(D_{2^\infty})$ and let $\varrho\in\mathrm{S}(B)$ be a state extending $\mathrm{ev}_n\otimes\psi$. If $\iota_n\colon M_{2^\infty} \hookrightarrow B$ denotes the $n$-th coordinate embedding, then $\varrho\circ\iota_n$ is a state on $M_{2^\infty}$ extending $\psi$, whence by the unique extension property of $(D_{2^\infty} \subset M_{2^\infty})$ we have $\varrho\circ\iota_n = \psi\circ E$. Also, it follows directly by its definition that $\varrho$ vanishes on sequences in $B$ where the $n$-th entry is zero. This shows that $\varrho$ is unique, so $\mathrm{ev}_n\otimes\psi$ extends uniquely to $B$ (and thus also to $A$).

Now let $\psi\in\mathrm{PS}(D_{2^\infty})$ and consider the pure state $\theta\coloneqq \mathrm{ev}_\infty\otimes\psi\in\mathrm{PS}(D)$. Clearly $\theta$ will fail to extend uniquely on $B$, since if $\omega\in\beta\bb{N}\setminus\bb{N}$ is any free ultrafilter, the state $(b_j)_{j\ge1}\mapsto\lim_{j\to\omega}\psi(E(b_j))$ will extend $\theta$ on $B$. Nevertheless, since $A=\mathrm{C}^\ast(D,p)$, if we show that any extension $\bar{\theta}$ of $\theta$ on $A$ vanishes at $p$, then since $p$ is a projection we will have that $0=|\bar{\theta}(p)|^2=\bar{\theta}(p^2)$, and so $p$ will lie in the multiplicative domain of $\bar\theta$. As $D$ also lies in the multiplicative domain of $\bar{\theta}$ since $\theta\in\mathrm{PS}(D)$, we have that $\bar{\theta}$ vanishes on any product of elements from $D\cup \{p\}$ whenever at least one term $p$ appears. Since $A=\mathrm{C}^\ast(D,p)$, this shows that the extension of $\theta$ is uniquely determined on $A$. Now since $\psi\in\mathrm{PS}(D_{2^\infty})$ and the spectrum of $D_{2^\infty}$ is the Cantor space $\{1,2\}^\bb{N}$, there is a point $x=(x_j)_{j\ge1}\in\{1,2\}^\bb{N}$ such that $\psi$ corresponds to evaluation at $x$. For $k\in\bb{N}$ set $\bar{x}_k:=(x_j)_{j=1}^k\in\{1,2\}^k$ and, for a $k$-tuple $\bar{y}:=(y_j)_{j=1}^k\in\{1,2\}^k$ set $e_{\bar{y}}:=e_{y_1,y_1}\otimes\dots\otimes e_{y_k,y_k}\in D_2^{\otimes k}\subset D_{2^\infty}$. Now let $e_k:=(0,\dots,0,e_{\bar{x}_k},e_{\bar{x}_k},\dots)\in D$, where the segment of zeroes in the beginning of $e_k$ has length $k$. Clearly  $\theta(e_k)=1$ for all $k\in\bb{N}$, so each $e_k$ lies in the multiplicative domain of $\theta$, hence $\theta(p)=\theta(e_kpe_k)$ for all $k\in\bb{N}$. We have that $\|e_kpe_k\|=\sup_{n\ge k+1}\|e_{\bar{x}_k}p_ne_{\bar{x}_k}\|_{M_2^{\otimes n}}$, and for $n\ge k+1$ we have
\begin{align*}
\|e_{\bar{x}_k}p_ne_{\bar{x}_k}\|_{M_2^{\otimes n}}&=\bigg\|\frac{1}{2^n}\sum e_{x_1,x_1}\otimes\dots\otimes e_{x_k,x_k}\otimes e_{i_{k+1},j_{k+1}}\otimes\dots\otimes e_{i_n,j_n}\bigg\| \\
&=\frac{1}{2^k},
\end{align*}
where the sum is taken over all $i_{k+1},j_{k+1},\dots,i_n,j_n\in\{1,2\}$. We thus conclude that $\|e_kpe_k\|=2^{-k}\to0$, so $\theta(p)=\lim_k\theta(e_kpe_k)=0$, as we wanted.

Since $(D\sub A)$ has the unique extension property, $D$ is a masa in $A$ and there exists a unique conditional expectation $\Phi\colon A\to D$, that also has the property that the unique extension of a pure state $\theta\in\mathrm{PS}(D)$ is given by $\theta\circ\Phi$ \cite[Remark~2.6, Corollary~2.7]{ArcBunGre82}. Evidently $\Phi$ is faithful: if $a=(a_j)_{j\ge1}\in A_+\subset\ell^\infty(\bb{N},M_{2^\infty})_+$ is such that $\Phi(a)=0$, then for $n\in\bb{N}$ and $\psi\in\mathrm{PS}(D_{2^\infty})$, we have $\psi(E(a_n))=(\mathrm{ev}_n\otimes\psi)\circ \Phi(a)=0$ and since $\psi$ is arbitrary this shows that $E(a_n)=0$ in $D_{2^\infty}$, hence $a_n=0$ in $M_{2^\infty}$ since $E$ is faithful, and since $n\in\bb{N}$ was arbitrary, we conclude that $a=0$. Note however that $\Phi$ is not hereditary: since $\Phi(p)$ is a convergent sequence in $D_{2^\infty}$, let $d\in D_{2^\infty}$ denote its limit. As explained earlier, for any $\psi\in\mathrm{PS}(D_{2^\infty})$ we have that $(\mathrm{ev}_\infty\otimes\psi)\circ\Phi(p)=0$, so $\psi(d)=0$ for all $\psi\in\mathrm{PS}(D_{2^\infty})$, hence $d=0$, which shows that $\Phi(p)\in C_0(\bb{N}, D_{2^\infty})$ and thus $\her(\Phi(p))\subset C_0(\bb{N}, M_{2^\infty})$. Clearly $p\not\in C_0(\bb{N},M_{2^\infty})$, so $p\not\in\her(\Phi(p))$.

Note that $(D\sub A)$ is non-regular, since $\mc{N}_A(D)=D$. Also, $A$ is an AF algebra, which we prove by showing that $A$ is locally AF. Since $A$ is generated by $D\cup\{p\}$, it suffices to show that if $\Omega$ is a finite subset of $D\cup\{p\}$ and $\varepsilon>0$, then there exists a finite-dimensional sub-$\mathrm{C}^\ast$-algebra $F\subset A$ such that $\mathrm{dist}(a,F)<\varepsilon$ for all $a\in\Omega$. 

Since each $d\in \Omega\cap D$ is a convergent sequence with entries in $D_{2^\infty}$ we obtain $k \in \bb{N}$ such that each $d \in D$ is approximated by a sum of a sequence with the first $k$ entries lying in $D_2^{\otimes k}$ and the rest of the entries being $0$, and a sequence with the first $k$ entries being $0$ and the rest of the entries being all the same and equal to some element in $D_2^{\otimes k}$.
Set 
\begin{equation*}
C_k \coloneqq \big\{(b_j)_{j\ge1}\in B : b_1,\dots,b_k\in M_2^{\otimes k}, b_j=0\text{ for }j\ge k+1\big\}\cong \bigoplus_{j=1}^kM_2^{\otimes k}
\end{equation*}
and $h_k \coloneqq 1_B-1_{C_k}\in D$ and for $\bar{x}\in\{1,2\}^k$ set $\tilde{e}_{\bar{x}} \coloneqq (0,\dots,0,e_{\bar{x}},e_{\bar{x}},\dots)\in D$ where the initial segment of zeroes here has length $k$. We claim that there exists a finite-dimensional $\mathrm{C}^*$-algebra $F_k\subset A$ with $1_{F_k} = h_k$ and such that, if $q_k:=h_kp=(0,\dots,0,p_{k+1},p_{k+2},\dots)\in A$, then
\begin{equation*}
\{\tilde{e}_{\bar{x}}:\bar{x}\in\{1,2\}^k\}\cup\{q_k\}\subset F_k.
\end{equation*}
In particular $F_k$ will contain every sequence of the form $(0,\dots,0,d,d,d,\dots)$ where $d\in D_2^{\otimes k}$, and so $F \coloneqq (C_k \cap A) + F_k$ will be a finite-dimensional sub-$\mathrm{C}^\ast$-algebra of $A$ that approximately contains $\Omega$, as explained above.

To prove the claim, for $\bar{x},\bar{y}\in\{1,2\}^k$ we set $h_{\bar{x},\bar{y}}:=2^k\tilde{e}_{\bar{x}}\cdot p\cdot \tilde{e}_{\bar{y}}$ and straightforward calculations show that $\{h_{\bar{x},\bar{y}}\}_{\bar{x},\bar{y}\in\{1,2\}^k}$ form a system of matrix units for $M_{2^k}$. We can thus define a (non-unital) embedding $\iota\colon M_{2^k}\hookrightarrow A$ by declaring $\iota(e_{\bar{x},\bar{y}}):=h_{\bar{x},\bar{y}}$ on the matrix units $\{e_{\bar{x},\bar{y}}\}_{\bar{x},\bar{y}\in\{1,2\}^k}$ of $M_{2^k}$ and extending linearly. It is clear that $\iota(1)\le h_k$, that $q_k=\iota(\frac{1}{2^k}\sum_{\bar{x},\bar{y}}e_{\bar{x},\bar{y}})$ and that $\tilde{e}_{\bar{x}}\cdot\iota(1)=h_{\bar{x},\bar{x}}=\iota(1)\cdot\tilde{e}_{\bar{x}}$. Set $g_{\bar{x}}:=(h_k-\iota(1))\tilde{e}_{\bar{x}}=\tilde{e}_{\bar{x}}-h_{\bar{x},\bar{x}}$ for all $\bar{x}\in\{1,2\}^k$. We have that $\{g_{\bar{x}}\}_{\bar{x}\in\{1,2\}^k}$ are pairwise orthogonal projections that are all orthogonal to $\iota(M_{2^k})$, hence $F_k:=\mathrm{span}\{g_{\bar{x}}:\bar{x}\in\{1,2\}^k\}+\iota(M_{2^k})\subset A$ is a sub-$\mathrm{C}^*$-algebra isomorphic to $\bb{C}^{2^k}\oplus M_{2^k}$. As we already noted, $q_k\in\iota(M_{2^k})\subset F_k$ and also, $\tilde{e}_{\bar{x}}=g_{\bar{x}}+h_{\bar{x},\bar{x}}\in F_k$ for all $\bar{x}\in\{1,2\}^k$, as we wanted.

\end{example}

\end{document}